\documentclass[a4paper]{amsart}
\author{Julien Melleray}
\address{Universit\'e Claude Bernard -- Lyon 1 \\
  Institut Camille Jordan, CNRS UMR 5208 \\
  43 boulevard du 11 novembre 1918 \\
  69622 Villeurbanne Cedex \\
  France}
\usepackage{amssymb}
\usepackage[initials,shortalphabetic]{amsrefs}
\usepackage[all]{xy}
\usepackage{mathpazo}
\usepackage{hyperref}
\usepackage{bm}
\usepackage{my-macros}
\numberwithin{equation}{section}
\usepackage{csquotes}
\usepackage{stmaryrd}
\newcommand{\Free}{\mathrm{Free}}

\title{Clopen type semigroups of actions on $0$-dimensional compact spaces}

\begin{document}
\begin{abstract}
We investigate properties of the clopen type semigroup of an action of a countable group on a compact, $0$-dimensional, Hausdorff space $X$. We discuss some characterizations of dynamical comparison (most of which were already known in the metrizable case) in this setting; and prove that for a Cantor minimal action $\alpha$ of an amenable group the topological full group of $\alpha$ admits a dense, locally finite subgroup iff the corresponding clopen type semigroup is unperforated. We also discuss some properties of clopen type semigroups of the Stone-\v{C}ech compactifications  and universal minimal flows of countable groups, and derive some consequences on generic properties in the space of minimal actions of a given countable group on the Cantor space.
\end{abstract}
\maketitle

%%%%%%%%%%%%%%%%%%%%%%%%%%%%%%%%%%%%%%%%%%%%%%%%%%%%%%%%%%%%%%%%%%%%%%%%%%%%%%%%%%%%%%%%%%%%%%%%%%%%%%%%%%%%%%%%%%%%%%
%%%%%%%%%%%%%%%%%%%%%%%%%%%%%%%%%%%%%%%%%%%%%%%%%%%%%%%%%%%%%%%%%%%%%%%%%%%%%%%%%%%%%%%%%%%%%%%%%%%%%%%%%%%%%%%%%%%%%%

\section{Introduction}

Given an action $\alpha \colon \Gamma \actson X$ of a countable group $\Gamma$ on a compact, Hausdorff, $0$-dimensional space $X$, and two clopen subsets $A,B$ of $X$, an interesting question is whether one can \emph{equidecompose} $A$ into $B$, that is, whether there exists a clopen partition $A=\bigsqcup_{i=1}^n A_i$ and $\gamma_1,\ldots,\gamma_n \in \Gamma$ such that $\bigsqcup_{i=1}^n \gamma_i A_i \subset B$. 

An important and well-studied particular case of this section occurs when $X$ is the Stone-\v{C}ech compactification $\beta \Gamma$ of $\Gamma$: the action $\Gamma \actson \beta \Gamma$ really is the action of $\Gamma$ on the powerset of $\Gamma$ in disguise. Equidecomposability problems in this setting have been studied since the beginning of the twentieth century, in particular in connection with the Banach-Tarski paradox and paradoxical decompositions. 

At the other extreme, the case where $X$ is the Cantor space and $\Gamma=\Z$ acts minimally on $X$ is also well-studied, and connected to classification of such actions up to flip-conjugacy and orbit equivalence as well as operator algebraic properties (see \cite{GPS_OE} or the book \cite{Putnam2018}). A key fact is that such actions have the \emph{dynamical comparison property} introduced in\cite{Buck2013} and \cite{Kerr2020}: if $\alpha \colon \Z \actson X$ is a minimal action on the Cantor space, and $A$, $B$ are clopen such that $\mu(A)< \mu(B)$ for every $\alpha$-invariant Radon probability measure, then a result of Glasner and Weiss \cite{GlasnerWeiss} asserts that one can equidecompose $A$ into $B$. The dynamical comparison property has been intensively studied in recent years following work of Kerr that established its relevance to problems related to operator algebras. Among notable recent results, we mention that Downarowicz and Zhang \cite{Downarowicz_Zhang} proved that every action of a locally subexponential group on a compact $0$-dimensional metrizable Hausdorff space has dynamical comparison; Naryshkin \cite{Naryshkin2021} proved the same result for actions of groups of polynomial growth on arbitrary compact metrizable spaces (though the definition of dynamical comparison we give here only applies to $0$-dimensional spaces, one can extend it to arbitrary compact metrizable spaces, see \cite{Kerr2020}). Kerr and Narishkyn \cite{Kerr_Narishkyn} showed that every free action of an elementary amenable group on a compact metrizable space has dynamical comparison.

Recall that the \emph{topological full group} $\llbracket \alpha \rrbracket$ of $\alpha \colon \Gamma \actson X$ is the group of all homeomorphisms $g$ such that there exist a clopen partition $X=\bigsqcup_{i=1}^n A_i$ and $\gamma_1,\ldots,\gamma_n \in \Gamma$ with $g(x)=\alpha(\gamma_i)(x)$ for all $x \in A_i$.  In the case of a minimal $\Z$-action $\alpha$ on the Cantor space, $\llbracket \alpha \rrbracket$ admits a natural dense, locally finite subgroup, made up of all the elements of $\llbracket \alpha \rrbracket$ which preserve the positive semi-orbit of some fixed element $x_0$; these dense locally finite subgroups play an important part in the Giordano--Putnam--Skau classification of minimal Cantor $\Z$-actions up to orbit equivalence (see \cite{GPS_OE} and \cite{GPS_fullgroups} for the original approach, or \cite{mellerayInvariantMeasuresOrbit2021} and its references for another point of view). Whether such a classification is valid for all countable amenable groups is an open problem: it is not known if every minimal Cantor action of a countable amenable group is orbit equivalent to a $\Z$-action. A weaker property is also not known to hold in general: given a minimal Cantor action $\alpha$ of a countable amenable group, does there exist a $\Z$-action which has the same invariant Borel probability measures as $\alpha$? If
$\alpha$ has dynamical comparison then it follows from a result of Ibarluc\'ia and the author (see \cite{Melleray2017}, \cite{Melleray_dynamicalsimplices}) that the answer to this question is positive (see Section \ref{s:measured_comparison}, where we reformulate this property as a weaker version of dynamical comparison). 

It was noted by Kerr in \cite{Kerr2020} that dynamical comparison for an action $\alpha \colon \Gamma \actson X$ is related to a property of its \emph{clopen type semigroup}; loosely speaking, this is the free monoid generated by nonempty clopen subsets of $X$, quotiented by the relation $\sum_{i=1}^n A_i \sim \sum_{j=1}^n B_j$ if one can equidecompose $\bigsqcup A_i \times \{i\}$ and $\bigsqcup B_j \times \{j\}$ using the action of $\Gamma \times \mathfrak S_n$ on $X \times \{1,\ldots,n\}$ (see the next section for more details). One then obtains a \emph{refinement monoid} which we denote $T(\alpha)$; for $a,b \in T(\alpha)$ we write $a \le b$ when there exists $c \in T(\alpha)$ such that $b=a+c$. This monoid is closely related to the $0$-homology group $H(\alpha)$ of the action; indeed we observe in Section \ref{s:homology} that if $T(\alpha)$ is \emph{cancellative} (i.e. such that $a+b=a+c \Rightarrow b=c$) then $T(\alpha)$ is precisely the positive cone of $H(\alpha)$.

Type semigroups were introduced by Tarski to study equidecomposition problems, and the clopen version was first considered in an article of R\o rdam and Sierakowski \cite{Rordam_Sierakowski}; \cite{Wagon2016} is a nice, modern reference on type semigroups and \cite{Wehrung2017} contains a wealth of information on refinement monoids. Kerr observed in \cite{Kerr2020} that a free minimal action $\alpha$ of a countable group $\Gamma$ on a Cantor space $X$ has dynamical comparison as soon as $T(\alpha)$ is \emph{almost unperforated}, i.e.~ one has for all $n \in \N$ and all $a,b \in T(\alpha)$ that $(n+1)a \le nb \Rightarrow a \le b$. Later, Ma proved in \cite{Ma2021} that, for minimal actions on compact metrizable spaces (even, not $0$-dimensional, using a more general definition of the type semigroup) almost unperforation and dynamical comparison are equivalent; even for non-minimal actions, Ma \cite{Ma2021} proved that, when the acting group is amenable, dynamical comparison is equivalent to a condition slightly weaker than almost unperforation. Recently, Ara, Bönicke, Bosa and Li established in \cite{Ara_Bonicke_Bosa_Li} a similar result, valid for second-countable ample groupoids. 

In this article, we note that this analysis can also be carried out when the ambient compact space is not metrizable (but is $0$-dimensional), and give an elementary proof of a characterization of dynamical comparison for such actions as a property of the clopen type semigroup. This leads in particular to the following result (due to Ma in the metrizable case, see \cite{Ma2019} and \cite{Ma2021}) .

\begin{theorem*}[see Lemma \ref{l:carac_comparison_unperforation}, Lemma \ref{l:noinvmeasure} and Proposition \ref{p:minimal_dynamical_unperforated}]~

Assume that $\Gamma \actson X$ is an action of a countable group $\Gamma$ on a compact, $0$-dimensional, Hausdorff space $X$. 
\begin{itemize}
    \item If there are no $\alpha$-invariant Radon probability measures on $X$, then $\alpha$ has dynamical comparison iff $\alpha$ is minimal and $a \le b$ for every nonzero $a,b \in T(\alpha)$.
    \item If $\Gamma$ is amenable, then $\alpha$ has dynamical comparison iff for every order unit $b \in T(\alpha)$ and every $a$ such that $(n+1)a \le nb$ for some $n \in \N$, one has $a \le b$.
    \item If $\alpha$ is minimal then $\alpha$ has dynamical comparison iff $T(\alpha)$ is almost unperforated.
\end{itemize}
\end{theorem*}
(one says that $b$ is an order unit in $T(\alpha)$ iff for every $a \in T(\alpha)$ there is some $k \in \N$ such that $a \le kb$)

There is a significant literature on refinement monoids (see \cite{Wehrung2017} and references therein), and some of it has interesting applications to our setting. In particular, this can be used to give another equivalent formulation of dynamical comparison when $\alpha$ is minimal, as well as conclude that if $\alpha$ is minimal, preserves some invariant Radon probability measure and has dynamical comparison then $T(\alpha)$ is \emph{cancellative}, i.e.~ whenever $a+b=a+c$ in $T(\alpha)$ one has $b=c$. Interestingly I do not know any direct argument to prove this, nor do I know whether clopen type semigroups of minimal Cantor actions of countable amenable groups are always cancellative. 

It is fairly easy to see that if $\llbracket \alpha \rrbracket$ admits a dense locally finite subgroup then $T(\alpha)$ is cancellative and \emph{unperforated}, i.e.~ $na \le nb \Rightarrow a \le b$ for every $n \in \N^*$ and every $a,b \in T(\alpha)$. Exploiting a connection between $T(\alpha)$ and the $0$-homology group of $\alpha$ (see section \ref{s:homology} for details), it then follows from a result of Matui \cite{Matui_torsion} that there exists a free, minimal action $\alpha$ of $\Z^2$ on the Cantor space such that $\llbracket \alpha \rrbracket$ does not have a dense locally finite subgroup. This stands in stark contrast to the case of minimal $\Z$-actions. It turns out that the existence of a dense locally finite subgroup of $\llbracket \alpha \rrbracket$ is captured by algebraic properties of $T(\alpha)$.

\begin{theorem*}[see Theorem~\ref{p:unperf_locally_finite_dense}]~

Assume that $\alpha$ is an action of a countable group $\Gamma$ on a compact, metrizable, $0$-dimensional space. Then $\llbracket \alpha \rrbracket$ admits a dense locally finite subgroup iff $T(\alpha)$ is unperforated and cancellative.
\end{theorem*}
When $\Gamma$ is amenable and $\alpha$ is minimal, it follows that $\llbracket \alpha \rrbracket$ has a dense locally finite subgroup iff $T(\alpha)$ is unperforated.

The fact that our arguments work also in the non-metrizable setting makes it easy to deduce the following facts from classical results.

\begin{theorem*}[See Theorem \ref{t:Tarski}]~

Let $\Gamma$ be a countable amenable group. Then the action $\Gamma \actson \beta \Gamma$ has dynamical comparison. 

This implies that for any two subsets $A,B$ of $\Gamma$, if $\mu(A) < \mu(B)$ for every finitely additive, $\Gamma$-invariant probability measure $\mu$ on $\Gamma$, then there exist $A_1,\ldots,A_n \subset \Gamma$ and $\gamma_1, \ldots , \gamma_n \in \Gamma$ such that $\bigsqcup_{i=1}^n A_i = A$, $\bigsqcup_{i=1}^n \gamma_i A_i \subseteq B$.
\end{theorem*}

Using a standard trick of topological dynamics (the existence of a $\Gamma$-equivariant retraction from $\beta \Gamma$ onto the \emph{universal minimal flow} $\mu \Gamma$ of $\Gamma$), we also obtain the following fact.

\begin{theorem*}[See Theorem \ref{t:minimal_flow_comparison}]~
Let $\Gamma$ be a countable group. Then the action $\Gamma \actson \mu \Gamma$ has dynamical comparison.
\end{theorem*}

We then use this to glean some insight into the generic properties of minimal actions of countable amenable groups on the Cantor space. For this to make sense, note first that the space of actions $A(\Gamma)$ on the Cantor space $X$ can be seen as a closed subspace of $\Homeo(X)^\Gamma$, hence a Polish space. The space of minimal actions $\Min(\Gamma) \subset A(\Gamma)$ is a $G_\delta$ subset, so it is also Polish with the induced topology; and the conjugation action $\Homeo(X) \actson \Min(\Gamma)$ is readily seen to be transitive. It follows that every Baire-measurable subset of $\Min(\Gamma)$ is either meager or comeager. This applies in particular to conjugacy classes, and one is led to wonder when there exists a comeager conjugacy class in $\Min(\Gamma)$. This is known to hold for $\Gamma=\Z$ by a result of Hochman (\cite{Hochman2008}; the generic element is the universal odometer) and seems to be an open problem for every other countable group $\Gamma$. As a consequence of our work, we establish the following fact.

\begin{theorem*}[see Proposition \ref{p:generic_properties}]~

Let $\Gamma$ be an infinite countable group. A generic element $\alpha$ of $\Min(\Gamma)$ is such that:
\begin{itemize}
    \item $\alpha$ is free.
    \item $T(\alpha)$ is unperforated (hence $\alpha$ has dynamical comparison).
    \item The algebraic order on $T(\alpha)$ is a partial ordering.
\end{itemize}
\end{theorem*}
This can be used to give an alternative argument for the fact, due to Conley, Kerr, Jackson, Marks, Seward and Tucker-Drob that a generic element of $\Min(\Gamma)$ is \emph{almost finite} \cite{Conley_etc2018}. 

The organization of the article is fairly straightforward: we begin by introducing the clopen type semigroup of an action, discuss how it is related to the $0$-homology group and see how certain properties of an action (such as dynamical comparison) are visible in its clopen type semigroup. Then we discuss what this implies for the Stone-\v{C}ech compactification and universal minimal flow of a given countable group $\Gamma$, before deriving some consequences for generic properties in the space of minimal actions. After that we briefly discuss two weakenings of dynamical comparison and collect a few open problems.

There are two appendices to the paper; the first one is devoted to the proof of a variation on an result of Krieger which is used in our characterization of the existence of a dense locally finite subgroup of $\llbracket \alpha \rrbracket$ by algebraic properties of $T(\alpha)$ (that proof is essentially the same as Krieger's original argument as exposed in \cite{mellerayInvariantMeasuresOrbit2021}). The second one is simply a glossary of monoid-theoretic terms used in the paper, so if the reader is uncertain what the precise definition of such a notion is, and unsure where it came up first, then she can simply refer to the second appendix.

%we note that some of the ideas discussed here can be used to improve a result of \cite{Melleray_dynamicalsimplices} about \emph{approximate divisibility} of sets of invariant measures, and that $m$-comparison (in the sense of Kerr \cite{Kerr2020}) is equivalent to dynamical comparison for aperiodic actions on $0$-dimensional metrizable compact spaces. 

As we already mentioned above, the results in section \ref{s:clopen_type_semigroup} concerning dynamical comparison were already known (and due to Ma, \cite{Ma2019} and \cite{Ma2021}) for metric spaces; we still give complete arguments since one of the aims of this paper is to gather some of the ideas used when studying refinement monoids and make them accessible to an audience not used to this topic (which includes the author before embarking on this project), as well as to point out some of the relevant literature. Our hope is that this article can serve as a starting point for a deeper study of these connections. 
  
\emph{Acknowledgements}. My interest in the clopen type semigroup in the non metrizable case stems from a suggestion of Andreas Thom. This work also benefited from conversations with Andy Zucker as well as Nicol\'as Matte Bon, who made me aware of some of the literature related to this topic. Mikael R\o rdam kindly provided a reference that helped me see how one can prove Proposition \ref{prop: almost unperforation condition iff states agree} and Benjamin Weiss explained to me an interesting example for $\Z$-actions that is briefly discussed in the paper. Andrea Vaccaro made helpful comments on a previous version of this paper. I would like to thank all of them.

Simon Robert should have a special mention here, for our many discussions regarding this paper; Simon eventually elected not to co-sign this article but those conversations were important in it coming to fruition. 

\section{Properties of the clopen type semigroup} \label{s:clopen_type_semigroup}

\subsection{Some vocabulary}

Throughout this section we fix a $0$-dimensional compact Hausdorff space $X$, a countable discrete group $\Gamma$, and an action $\alpha \colon \Gamma \curvearrowright X$. Since $\alpha$ is fixed for now, we suppress it from our notation and simply write $\gamma x$ for $\alpha(\gamma)(x)$.

We consider the space $Y = X \times \N$, endowed with the product topology, and say that $A \in \Clopen(Y)$ is \emph{bounded} if $A \cap (X \times \{n\}) = \emptyset$ for any large enough $n$. We denote $\tilde \Gamma = \Gamma \times \mathfrak S$, where $\mathfrak S$ is the permutation group of $\N$, and let $\Gamma$ act diagonally on $Y$.

\begin{defn}
Let $A,B$ be two bounded clopen subsets of $Y$.

We say that $A$ and $B$ are \emph{equidecomposable} if there exist $A_i \in \Clopen(Y)$, $\tilde \gamma_i \in \tilde \Gamma$ such that $A= \bigsqcup_{i=1}^n A_i$, $B= \bigsqcup_{i=1}^n \tilde \gamma_i A_i $.

We denote by $[A]$ the set of all bounded clopen subsets of $Y$ which are equidecomposable with $A$, and set $T(\alpha)= \{[A] \colon A \text{ bounded and clopen in } Y\}$.
\end{defn}

\begin{defn}
Given two bounded clopen $A,B \subset Y$ we let $[A] + [B] = [\tilde A \cup \tilde B]$, where $\tilde A$ and $\tilde B$ are any two bounded clopen subsets of $Y$ 
such that $[\tilde A]= [A]$, $[\tilde B]= B$, and $\tilde A \cap \tilde B= \emptyset$.
\end{defn}

It is straightforward to check that this definition indeed makes sense ($[\tilde A \sqcup \tilde B]$ does not depend on the choice of $\tilde A \in [A]$, $\tilde B \in [B]$ as long as they are disjoint) and that $(T(\alpha),+)$ is a commutative monoid with neutral element $0=[\emptyset]$. 

We endow $T(\alpha)$ with the algebraic pre-ordering, i.e.~set $a \le b$ whenever there exists $c \in T(\alpha)$ such that $a+c=b$. The structure $(T(\alpha),+,\le)$ is the \emph{clopen type semigroup} with which we will be working throughout this paper. This construction is analogous to Tarski's well-known approach to paradoxicality questions, see \cite{Wagon2016}; the clopen version discussed here first came up in \cite{Rordam_Sierakowski}.

\begin{defn}
A commutative monoid $T$ is 
\begin{itemize}
    \item \emph{conical} if for any $u,v \in T$ one has $u+v=0 \Rightarrow u=v=0$.
    \item \emph{simple} if for any $y$, any $x \ne 0$ there exists $n$ such that $y \le nx$ (every element is an \emph{order unit}).
    \item a \emph{refinement monoid} if whenever $\sum_{i=1}^n a_i = \sum_{j=1}^m b_j$, there exist $(c_{i,j})$ such that for all $i$ one has $\sum_{j=1}^m c_{i,j}=a_i$, and for all $j$ one has $\sum_{i=1}^n c_{i,j} = b_j$.
\end{itemize}
\end{defn}

Clearly, $T(\alpha)$ is a conical refinement monoid; and $T(\alpha)$ is simple iff $\alpha$ is \emph{minimal}, i.e.~if there is no nontrivial closed $\alpha$-invariant set (equivalently, every orbit of $\alpha$ is dense in $X$).

\subsection{Connection with $0$-homology and cancellativity}\label{s:homology}
We now take some time to discuss the strong connection between the clopen type semigroup and the $0$-homology group of the action. 

To that end, we use an alternative construction of $T(\alpha)$ mentioned by Kerr \cite{Kerr2020}: given $f \in C(X,\N)$ and $\gamma \in \Gamma$, let $\gamma \cdot f $ be the function mapping $x \in X$ to $f(\gamma^{-1} x)$ (in this paper $\N=\{n \in \Z \colon n \ge 0\}$). Then, for any $f,g \in C(X,\N)$ say that $f \sim g$ if there exist $h_i \in C(X,\N)$ and $\gamma_i \in \Gamma$ such that $\sum_{i=1}^n h_i = f$ and $\sum_{i=1}^n \gamma_i \cdot h_i = g$. This is an equivalence relation and we denote by $[f]$ the equivalence class of $f \in C(X,\N)$. The addition on $C(X,\N)$ induces a well-defined, associative, commutative operation on $C(X,\N)/\!\! \sim$; and the monoid we just described is isomorphic to $(T(\alpha),+)$ via the (quotient of) the map $f \in C(X,\N) \mapsto \bigcup_{i=1}^n \{x \colon f(x) > i\} \times \{i\}$. 

Given a group action $\alpha \actson X$, $X$ a compact, $0$-dimensional space, the associated $0$-homology group, called the group of \emph{coinvariants} in \cite{Matui_torsion} is the group 
$$H(\alpha)=C(X,\Z) / \langle f - f \circ \alpha(\gamma) \colon \gamma \in \Gamma \rangle$$
We denote by $(f)$ the equivalence class of $f$ for this quotient map. 
Since $\alpha$ is fixed we again suppress if from our notation below and simply write $\gamma \cdot f$ for the map $x \mapsto f(\alpha(\gamma)^{-1}x)$.

If $[f]=[g]$ then there exist $h_1,\ldots, h_n \in C(X,\N)$ and $\gamma_1,\ldots,\gamma_n \in \Gamma$ such that
\[f = g+ \sum_{i=1}^n (h_i - \gamma_i \cdot h_i) \]
so that $[f]=[g] \Rightarrow (f)=(g)$. 
So we have a natural surjection $\pi \colon T(\alpha) \to H(\alpha)^+$, where $H(\alpha)^+$ is the positive cone of $H(\alpha)$, i.e.~the image of $C(X,\N)$ under the quotient map; by definition $\pi$ is a semigroup homomorphism from $(T(\alpha),+)$ to $(H(\alpha),+)$ and we are then led to wonder when $\pi$ is injective.

\begin{defn}
A commutative semigroup $S$ is \emph{cancellative} if for any $a,b,c \in S$ we have $(a+b = a + c) \Rightarrow b= c$.
\end{defn}

Since $H(\alpha)$ is a group, it is of course cancellative; so if $\pi$ is injective then $T(\alpha)$ must itself be cancellative.

\begin{prop}
The natural semigroup homomorphism $\pi \colon T(\alpha) \to H(\alpha)$ is injective iff $T(\alpha)$ is cancellative.
\end{prop}

\begin{proof}
One implication is obvious and has already been pointed out. 

To see the converse, assume that $T(\alpha)$ is cancellative and $f,g \in C(X,\N)$ are such that $(f)=(g)$. Unravelling the definition, we obtain some $h_i \in C(X,\Z)$ and $\gamma_i \in \Gamma$ such that $f=g + \sum_{i=1}^n (h_i - \gamma_i \cdot h_i)$.

Since $h_i$ is continuous it is bounded, so for each $i$ there exists some integer $n_i$ such that $\tilde h_i  = h_i+ n_i \mathbf{1} \in C(X,\N)$ (where $\mathbf{1}$ is the constant function equal to $1$); and since $\gamma_i \cdot \mathbf 1= \mathbf 1$ we have that $\tilde h_i - \gamma_i \cdot \tilde h_i = h_i - \gamma_i \cdot h_i$.

We thus have $f + \sum_{i=1}^n \gamma_i \cdot \tilde h_i = g + \sum_{i=1}^n \tilde h_i$; and by definition $[\gamma_i \cdot \tilde h_i] = [ \tilde h_i]$ so we obtain $[f] + [\sum_{i=1}^n \tilde h_i] = [g] + [\sum_{i=1}^n [\tilde h_i]$. Since we assumed that $T(\alpha)$ is cancellative we conclude that $[f]=[g]$, in other words $\pi$ is injective.
\end{proof}

Using a similar argument one can show that $H(\alpha)$ is always isomorphic to the Grothendieck group of $T(\alpha)$.
It then becomes very interesting to understand when $T(\alpha)$ is cancellative, especially so because as soon as $T(\alpha)$ is cancellative it is the positive cone of a group satisfying the Riesz decomposition and interpolation properties (which in this situation are equivalent to refinement, see \cite{Goodearl2010}*{Prop 2.1}), and groups satisfying these conditions are well-studied.

Unfortunately I know little about this question, though below (Proposition \ref{prop: link between weak comparability, almost unperforation and cancellation}) a connection with dynamical comparison will appear.

In section \ref{s:Stone_Cech} we will see that the clopen type semigroup of the universal minimal flow of any countable amenable group is cancellative. In the other direction, let us discuss a simple example of a uniquely ergodic, free Cantor $\Z$-action such that $T(\alpha)$ is not cancellative. Note that this action has the dynamical comparison property since the acting group is $\Z$.

Let $Y=\Z \cup \{\infty\}$ be the one-point compactification of $\Z$, with $\Z$ acting on itself by translation and fixing $\infty$. Form the product $X= Y \times 2^\omega$, with $\Z$ acting on the second coordinate via (say) the dyadic odometer. Then $X$ is a Cantor space, on which $\Z$ acts freely; further, there is a unique $\Z$-invariant Borel probability measure, supported on $\{\infty\} \times 2 ^\omega$. Let $A = (Y \setminus \{0\}) \times 2^\omega$; the map 
\[ \tau \colon (n,x) \mapsto \begin{cases} (n+1,x) & \text{ if } n \ge 0 \\ (n,x) & \text{ if } n < 0 \end{cases} \]
witnesses that $[Y] \le [A] $, so $T(\alpha)$ is not cancellative (in the language introduced in the next section, it is not even stably finite). 

% (there exists a topologically transitive action with these properties by the standard argument, since $\beta \Z$ is not stably finite; probably if we take the closure of the orbit of some $(0,x)$ we get such an example).

%Also note that $B = \{0\} \times 2^\omega$ is clopen, has measure $0$ for every invariant Borel probability measure, yet one does not have $(n+1)[B] \le n [B]$ for any $n$. Mention in the next section? With reorganization not natural to do that here.

% We recall here that by Ara-Pardo REF, in the minimal, amenable case the dynamical comparison property (in the form of weak comparison) implies cancellation. So the existence of a non-cancellative minimal action of an amenable group appears unlikely (it would give an example of a group with a minimal action not satisfying dynamical comparison!). 

There is another reason why cancellativity would be useful: we would like to rule out the existence of a minimal amenable Cantor action $\alpha$ and two clopen $A,B$ such that $[A]=[B]$ in $T(\alpha)$ yet there does not exist any $\gamma \in \llbracket \alpha \rrbracket$ such that $\gamma A = B$. If $[A \setminus B]= [B \setminus A]$ (which follows from cancellativity and the equality $[A]=[B]$) then there is an involution  $\sigma \in \llbracket \alpha \rrbracket$ such that $\sigma(A \setminus B) = B \setminus A$ and $\sigma$ fixes all the other elements, so that $\sigma A= B$; but without some form of cancellativity it is not clear whether $[A]=[B]$ implies $[A \setminus B]= [B \setminus A]$ in $T(\alpha)$. 

%Actually, it would be enough to know that $[A]=[B] \Rightarrow [X \setminus A]=[X \setminus B]$ to conclude that the clopens equidecomposable with some fixed clopen $A$ consist of all the translates of $A$ by elements of $\llbracket \alpha \rrbracket$.

\subsection{Existence of invariant measures}
\begin{defn}
A \emph{state} on $(T(\alpha),+)$ is a morphism from $(T(\alpha),+)$ to $[0,+ \infty]$, where as usual $x + \infty =+ \infty$ for every $x \in [0,+ \infty]$. A \emph{normalized} state is a state $\mu$ such that $\mu([X])=1$ (note that then $\mu$ only takes finite values).

We denote by $\mcM_\infty(\alpha)$ the set of states on $T(\alpha)$, and by $\mcM(\alpha)$ the set of nor\-ma\-lized states. 

\end{defn}

Note that normalized states correspond to $\Gamma$-invariant finitely additive probability measures (f.a.p.m) on $\Clopen(X)$: given a state $\mu$ one obtains a $\Gamma$-invariant f.a.p.m $\nu$ by setting $\nu(A) = \mu([A])$; conversely any $\Gamma$-invariant f.a.p.m $\nu$ induces a state by setting $\nu([A])= \mu(A)$ and extending to $T(\alpha)$ (this is straightforward to prove, see e.g. \cite{Wagon2016} for the classical case, and \cite{Rordam_Sierakowski} for the clopen case we consider here). In turn, since $X$ is $0$-dimensional, any $\Gamma$-invariant f.a.p.m on $\Clopen(X)$ uniquely determines a positive linear functional on $C(X)$ mapping $1$ to $1$, i.e.~a $\Gamma$-invariant Radon probability measure on $X$ (when $X$ is metrizable every Borel probability measure is Radon so we are simply dealing with Borel probability measures in that case); and a $\Gamma$-invariant Radon probability measure restricts to a $\Gamma$-invariant f.a.p.m. We will use the same notation for all these objects, making the obvious identifications whenever convenient. One should however beware that elements in $\mcM_\infty(\alpha)$ do not necessarily come from a Radon measure (a trivial example being the state mapping every non-zero element to $+\infty$; in the non-metrizable case there are other complications, such as the fact that $\Clopen(X)$ does not in general generate the Borel $\sigma$-algebra).

\begin{lemma}\label{l:order_unit_carac}
Assume that $\Gamma$ is amenable. Then $b \in T(\alpha)$ is an order unit iff $\mu(b)>0$ for every $\mu \in \mcM(\alpha)$.
\end{lemma}

\begin{proof}
Assume first that $b$ is an order unit. Then there exists $n$ such that $[X] \le nb$, whence $n \mu(b) \ge \mu(X)=1$ for every $\mu \in \mcM(\alpha)$.

Conversely, assume that $b$ is not an order unit; let $B=\bigsqcup B_i \times \{i\}$ be a representative of $b$ and let $A= \bigcup B_i$. Then $A$ is clopen and $\Gamma A \subsetneq X$ since otherwise a finite union of translates of the $B_i$ would cover $X$, which would imply that $[X] \le n b$ for some $n$. Thus $X \setminus \Gamma A$ is a nonempty closed subset of $X$, and since $\Gamma$ is amenable there exists a $\Gamma$-invariant Radon probability measure $\mu$ on $X$ which is supported on $X \setminus A$, hence $ \mu(B_i)=0$ for all $i$. Viewing $\mu$ as an element of $\mcM(\alpha)$, we have $\mu(b)=0$.
\end{proof}

\begin{defn}
We say that an element $b \in T(\alpha)$ is \emph{directly finite} if for all $a$ one has $a+b=b \Rightarrow a=0$. The semigroup $T(\alpha)$ is said to be \emph{stably finite} if every element of $T(\alpha)$ is directly finite. 
\end{defn}

Clearly cancellativity implies stable finiteness; we will shortly see that, in the minimal case, stable finiteness is equivalent to the existence of a normalized state on $T(\alpha)$. 

The next result follows immediately from Tarski's famous theorem \cite{Tarski1938}. See \cite{Wagon2016} for a detailed discussion of this and related results.

\begin{theorem}[Tarski, see \cite{Wagon2016}*{Theorem~11.1}]
Let $b \in T(\alpha)$. There exists $\mu \in \mcM_\infty(\alpha)$ such that $\mu(b)=1$ iff for any $n$ one does not have $(n+1)b \le nb$.
\end{theorem}

\begin{prop}
Assume that $\alpha$ is minimal. Then $\mcM(\alpha)$ is nonempty iff $T(\alpha)$ is stably finite.
\end{prop}

\begin{proof}
If there exists $\mu \in \mcM(\alpha)$, then for every nonzero element $a \in T(\alpha)$ we have $\mu(a)>0$ (because for some $n$ one has $[X] \le na$ since $a$ is an order unit). From $a+b=b$ we obtain $\mu(a)+\mu(b)=\mu(b)$, whence $\mu(a)=0$ since $\mu(b)$ is finite; and we noted that this is only possible if $a=0$.

Conversely, Tarski's theorem asserts that the nonexistence of a state on $T(\alpha)$ such that $\mu(X)=1$ implies that there exists $n$ such that $(n+1)[X] \le n[X]$, i.e. there exists $u$ such that $n[X] +[X]+u = n[X]$, whence $n[X]$ is not directly finite.
\end{proof}

In section \ref{s:homology} we saw an example of a uniquely ergodic $\Z$-action $\alpha$ on the Cantor space $X$ such that $(T(\alpha),+)$ is not stably finite.
Note that if $T(\alpha)$ is stably finite (in particular, if $\alpha$ is a minimal action of an amenable group) then $\le$ is a partial order on $T(\alpha)$; in general $\le$ may only be a pre-order, i.e. it could happen that $u \le v$ and $v \le u$ but $u \ne v$. For instance, it follows from the homology computation in the last section of \cite{Matui_TopfullGroups_subshifts} that for the natural action of a nonabelian free group on its boundary one has $a \le b$ for all nonzero $a,b \in T(\alpha)$, yet there exist nonempty clopen sets which are not equidecomposable. 

We now note the following fact, which (in our context) is the same as Proposition 2.1 of \cite{Ortega_Perera_Rordam2012}. This proposition is also used by Ara--Bönicke--Bosa--Li \cite{Ara_Bonicke_Bosa_Li}, Kerr \cite{Kerr2020} and Ma \cite{Ma2021}, among others. The proof given in \cite{Ortega_Perera_Rordam2012} appeals to results of Goodearl-Handelman \cite{Goodearl_Handelman1976} which are stated for partially ordered abelian groups, and it was not immediately clear to the author how to fill out all the details. M. R\o rdam was kind enough to point to Proposition 2.8 of \cite{Blackadar_Rordam1992} for an alternative argument to plug in to the proof; for the reader's convenience, here is a detailed argument.

\begin{prop}\label{prop: almost unperforation condition iff states agree}
For any $a,b \in T(\alpha)$, the following are equivalent:
\begin{enumerate}
\item There exists $n \in \N$ such that $(n+1)a \le nb$.
\item There exists $n$ such that $a \le n b$, and for every $\mu \in \mcM_\infty(\alpha)$ such that $\mu(b)=1$ one has $\mu(a) < \mu(b)$.
\end{enumerate}
\end{prop}

\begin{proof}
One implication is immediate. For the converse, fix $a,b \in T(\alpha)$ such that $a \le nb$ for some $n$ and $\mu(a)< \mu(b)$ for every $\mu \in M_{\infty}(\alpha)$ satisfying $\mu(b)=1$. 

We let $U_b$ denote the subsemigroup of $T(\alpha)$ made up of all $x$ such that $x \le nb$ for some $n$. Every state on $T(\alpha)$ restricts to a state on $U_b$, and every state on $U_b$ extends to a state on $T(\alpha)$ by assigning the value $\infty$ for every element not in $U_b$. 

Since $T(\alpha)$ may fail to be stably finite, we then form the maximal quasi-ordered quotient $V_b$ of $U_b$, declaring that $u \sim v$ iff $u \le v \le u$. To each state on $U_b$ corresponds a unique state on $V_b$, and states on $V_b$ extends to states on $U_b$.

Next, consider the Grothendieck group $G(V_b)$ of $V_b$, i.e the group obtained from $V_b$ by adding a formal inverse $-u$ for every $u \in V_b$; the natural map $\varphi \colon V_b \to G(V_b)$ may not be injective, since one has $\varphi(u)=\varphi(v)$ iff there exists $w \in V_b$ such that $u+w=v+w$ and we do not assume $T(\alpha)$ to be cancellative.

The group $G(V_b)$ becomes a partially ordered abelian group when declaring its positive cone to be equal to $\{u - v \colon u,v \in V_b\, , \ v \le u\}$. To avoid confusion we denote by $\preceq$ the corresponding ordering, and note that for any $u,v \in V_b$ we have $\varphi(u) \preceq \varphi(v)$ if, and only if, there exists $w$ such that $u + w \le v + w$.

Every normalized state on $G(V_b)$ (i.e. such that $\mu(b)=1$) restricts to a normalized state on $V_b$; and every normalized state on $V_b$ uniquely extends to a normalized state on $G(V_b)$. So our assumption on $a,b$ amounts to the statement that $\mu(\varphi(a)) < \mu(\varphi(b))$ for every normalized state on $G(V_b)$. Applying \cite{Goodearl2010}*{Theorem 4.12} it follows that $\varphi(b)-\varphi(a)$ is an order-unit in $G(V_b)$; in particular, there exists some integer $m$ such that $\varphi(a) \preceq m(\varphi(b)-\varphi(a))$.

Thus, there exists $u \in U_b$ and an integer $m$ such that $(m+1)a + u \le m b +u$. Using commutativity and associativity of $+$, we obtain $r(m+1)a +u \le rmb +u$ for any integer $r$; since $u \in U_b$ we have some integer $p$ such that $u \le pb$, leading to the inequality $r(m+1)a \le (rm+p)b$. Choose an integer $r >p$ and set $n=rm+p$. Then $(n+1)a \le r(m+1)a \le (rm+p)b=nb$.
\end{proof}
 
\subsection{Almost unperforation and dynamical comparison} 
 
 \begin{defn}
 Say that $T(\alpha)$ is \emph{almost unperforated} if for any $a,b \in T(\alpha)$ one has 
$$ \left( (n+1)a \le n b \right) \Rightarrow (a \le b)$$ 
 \end{defn}

\begin{lemma}\label{l:carac_comparison_unperforation}
$T(\alpha)$ is almost unperforated iff for any $a,b \in T(\alpha)$ one has
$$\left(\exists k \ a \le k b \text{ and } (\forall \mu \in \mcM_\infty(\alpha) \  \mu(b)=1 \Rightarrow \mu(a) < 1 )\right) \Rightarrow  \left( a \le b \right) $$
\end{lemma}

\begin{proof}
Assume $T(\alpha)$ is almost unperforated and $a,b \in T(\alpha)$ are such that $a \le kb$ for some $k$, and $\mu(a) < 1$ for every $\mu \in \mcM_\infty(\alpha)$ satisfying $\mu(b)=1$. Then by Proposition \ref{prop: almost unperforation condition iff states agree} there exists $n$ such that $(n+1)a \le n b$, whence $a \le b$ by almost unperforation.

The converse is immediate.
\end{proof}

If $\alpha$ is minimal the first condition on $a,b$ in Lemma \ref{l:carac_comparison_unperforation} is redundant, since it is satisfied for every order unit $b$ and every nonzero element is an order unit; and (still for minimal $\alpha$) the second condition amounts to the fact that $\mu(a) < \mu(b)$ for every $\mu \in \mcM(\alpha)$.

The property above is called the \emph{stable dynamical comparison property} in \cite{Ara_Bonicke_Bosa_Li}, and \emph{generalized dynamical comparison} in \cite{Ma2021}; we reserve ``dynamical comparison'' for another condition which we introduce now, since it seems to be commonly accepted terminology.

\begin{defn}[Buck \cite{Buck2013}, Kerr \cite{Kerr2020}]
We say that $\alpha$ has the \emph{dynamical comparison property} if for any two nonempty clopen $A,B \subseteq X$ such that $\mu(A)< \mu(B)$ for all $\mu \in \mcM(\alpha)$ one has $[A] \le [B]$.
\end{defn}

We ask that $B$ be nonempty above to rule out trivial counterexamples, for instance without that requirement no action $\alpha$ with $\mcM(\alpha)=\emptyset$ could have dynamical comparison, although this case leads to interesting questions. We will see shortly that, when $\mcM(\alpha)=\emptyset$, dynamical comparison is only possible if $\alpha$ is minimal.

We note the following fact, which shows that dynamical comparison is a property of $T(\alpha)$. The proof of Proposition 2.10 in \cite{Ara_Bonicke_Bosa_Li} directly gives this; that argument is actually in large part the same as the one used to show Lemma 3.2 in \cite{Downarowicz_Zhang}. We briefly describe the proof for the reader's convenience.

\begin{prop}\label{p:stable_comparison}
The action $\alpha$ has dynamical comparison if, and only if, for any nonzero $a,b \in T(\alpha)$ such that $\mu(a) < \mu(b)$ for all $\mu \in \mcM(\alpha)$ one has $a \le b$.
\end{prop} 

\begin{proof}
The implication from right to left is immediate. 

For the converse, fix an action $\alpha$ with dynamical comparison; let $a,b$ be such that $\mu(a) < \mu(b)$ for all $\mu \in \mcM(\alpha)$ then find some $p$ such that $a,b \le p[X]$, and choose representatives $A,B$ of $a,b$ contained in $X\times\{1,\ldots,p\}=\tilde X$.

The group $\tilde \Gamma = \Gamma \times \mathfrak S_{p}$ acts on $\tilde X$ via the action $\tilde \alpha \colon (\gamma,\sigma) \cdot (x,i)= (\gamma x, \sigma(i))$. We may, and do, identify $T(\tilde \alpha)$ and $T(\alpha)$.

For every $\mu \in \mcM(\tilde \alpha)=\mcM(\alpha)$ we have $\mu(A) < \mu(B)$, so by compactness of $\mcM(\alpha)$ there exists $\varepsilon \in \, ]0,1[$ such that $\mu(A)+ \varepsilon < \mu(B)$ for every $\mu \in \mcM(\alpha)$.

Enumerate $\tilde \Gamma = \{\tilde \gamma_i \colon i \in \N\}$ then define $A_0 = A \cap \tilde \gamma_0^{-1}(B)$, $B_0=\tilde \gamma_0(B)$ and, for $n \ge 1$, $A_n = (A \setminus \bigcup_{i=0}^{n-1} A_i) \cap \tilde \gamma_n^{-1} (B \setminus \bigcup_{i=0}^{n-1} \tilde \gamma_i A_i)$. 

The key step of the proof is the fact that, for large enough $n$, $\mu(A \setminus \bigcup_{i=0}^n A_i) < \frac{\varepsilon}{p}$ for all $\mu \in \mcM(\alpha)$. Grant this for the moment. Using the fact that $\alpha$ has dynamical comparison and $\mu(A \setminus \bigcup_{i=0}^n A_i) < 1$ for all $\mu \in \mcM(\alpha)$, we see that there exists a clopen subset $A_\infty$ of $X$ such that $[A_\infty]=[A \setminus \bigcup_{i=0}^ n A_i]$. 

Next, let $C= B \setminus \bigsqcup_{i=0}^n \tilde \gamma_i A_i$. By definition of $\varepsilon$, $\mu(C) \ge \varepsilon$ for all $\mu \in \mcM(\alpha)$; also, there exists $D \in \Clopen(\tilde X)$ such that $[D]=[C]$, $D=\bigcup_{i=1}^p D_i \times \{i\} \in \Clopen(\tilde X)$ and $D_{i+1} \subseteq D_i$ for all $i \le p-1$. In particular $\mu(D_1) \ge \frac{\mu(C)}{p} > \mu(A_\infty)$ for all $\mu \in \mcM(\alpha)$, so by dynamical comparison $[A_\infty] \le [D_1]$. We conclude by noting that 
% \begin{eqnarray*}
% a &=& \sum_{i=1}^n [A_i] + [A_\infty] \\
%     &=& \sum_{i=1}^n [\tilde \gamma_i A_i] + [A_\infty]\\
%     &\le& \sum_{i=1}^n [\tilde \gamma_i A_i] + [D_1] \\
%     &\le& \sum_{i=1}^n [\tilde \gamma_i A_i] + [C] \\
%     &\le& b
% \end{eqnarray*}

\[ a = \sum_{i=1}^n [A_i] + [A_\infty] \le \sum_{i=1}^n [\tilde \gamma_i A_i] + [D_1] \le \sum_{i=1}^n [\tilde \gamma_i A_i] + [C] \le b \]

To finish the proof, we go back to our claim that  $\mu(A \setminus \bigcup_{i=0}^n A_i) < \frac{\varepsilon}{p}$ for all $\mu \in \mcM(\alpha)$ as soon as $n$ is large enough. To see this, begin by fixing an ergodic measure $\nu \in \mcM(\alpha)$. If $\nu(A \setminus \bigcup_{i=0}^{\infty} A_i) \ne 0$, then by ergodicity there exists $j$ such that $(A \setminus \bigcup_{i=0}^{\infty} A_i) \cap \tilde \gamma_j^{-1}(B \setminus \bigcup_{i=0}^{+ \infty}  \tilde \gamma_i A_i) \ne \emptyset$, a contradiction. 

Thus $\mu(A \setminus \bigcup_{i=0}^{\infty} A_i)=0$ for all $\mu \in \mcM(\alpha)$.
Hence the maps $\mu \mapsto \mu( A \setminus \bigcup_{i=0}^{n} A_i)$ form a decreasing sequence of continuous maps from the compact space $\mcM(\alpha)$ to $[0,1]$ which converges pointwise to $0$. Applying Dini's theorem we obtain that the convergence is uniform, which gives the desired result.
\end{proof}

The following fact is also worth noting; for metrizable $X$ an equivalent statement appears as Corollary 3.12 of \cite{Ma2021}. Actually Ma's statement is proved for a more general definition of the type semigroup and applies also when $X$ is not $0$-dimensional (but metrizable).

\begin{lemma}\label{l:dyn_comp_unperf_below_order_unit}
Assume $\Gamma$ is amenable. Then the following are equivalent:
\begin{itemize}
    \item The action $\alpha$ has dynamical comparison.
    \item For every order unit $b \in T(\alpha)$, and every $a$ and $n$ such that $(n+1)a \le nb$, one has $a \le b$.
\end{itemize}
\end{lemma}

\begin{proof}
Assume that $\alpha$ has dynamical comparison, and $a,b \in T(\alpha)$ are such that $b$ is an order unit and $(n+1)a \le nb$ for some $n$. Then  $\mu(b)>0$ for all $\mu \in \mcM(\alpha)$ since $b$ is an order unit, whence $\mu(a)< \mu(b)$ for all $\mu \in \mcM(\alpha)$. So dynamical comparison gives, as desired, that  $a \le b$.

Conversely, assume that the second condition holds, and that $a,b \in T(\alpha)$ are such that $\mu(a) < \mu(b)$ for all $\mu \in \mcM(\alpha)$. Then Lemma \ref{l:order_unit_carac} implies that $b$ is an order unit, so $a \le nb$ for some $n$ and $\mu(a) < \mu(b)$ for every $\mu \in \mcM_\infty(\alpha)$ such that $\mu(b)=1$. Proposition \ref{prop: almost unperforation condition iff states agree} then implies that $(n+1)a \le nb$ for some $n$, whence $a \le b$ and $\alpha$ has dynamical comparison.
\end{proof}

\begin{lemma}\label{l:noinvmeasure}
Assume that $\mcM(\alpha)= \emptyset$. Then $\alpha$ has dynamical comparison iff $\alpha$ is minimal and $a \le b$ for every nonzero $a,b \in T(\alpha)$.
\end{lemma}

\begin{proof}
The implication from right to left is immediate.

For the converse implication, assume that $\mcM(\alpha) = \emptyset$ and $\alpha$ has dynamical comparison; then Proposition \ref{p:stable_comparison} implies that $a \le b$ for all $a,b \in T(\alpha)$. Fix a nonempty clopen $U$; we must have $[X] \le [U]$, whence $X$ is covered by translates of $U$. This proves that $\alpha$ is minimal.
\end{proof}

In the case where $\mcM(\alpha) = \emptyset$, or $\Gamma$ is amenable and $\alpha$ is free, and $X$ is the Cantor space, the following result appears in \cite{Ma2019}; in the second-countable case this is in \cite{Ara_Bonicke_Bosa_Li} (for ample groupoids). This statement also has a precursor in \cite{Kerr2020}*{Lemma~13.1}.

\begin{prop}\label{p:minimal_dynamical_unperforated}
Assume that $\alpha$ is minimal. Then $\alpha$ has dynamical comparison if, and only if, $T(\alpha)$ is almost unperforated.
\end{prop}

\begin{proof}
Assume first that $\mcM(\alpha) \ne \emptyset$. By minimality, if $\mu \in \mcM_\infty(\alpha)$ is such that $\mu(b)=1$ for some $b$ then also $\mu(X) < \infty$; so Lemma \ref{l:carac_comparison_unperforation} combined with Proposition \ref{p:stable_comparison} give the desired result.

Now, assume that $\mcM(\alpha)=\emptyset$. If $\alpha$ has dynamical comparison then $a \le b$ for any nonzero $a,b \in T(\alpha)$ by Lemma \ref{l:noinvmeasure}, hence $T(\alpha)$ is almost unperforated. To prove the converse, assume that $T(\alpha)$ is almost unperforated. Since $\alpha$ is minimal and $\mcM(\alpha)= \emptyset$, by Tarski's theorem for every nonzero $a$ there exists $n$ such that $(n+1)a \le na$. Using associativity, we obtain $(n+1) 2a = 2(n+1)a  \le na$. Almost unperforation then gives $2 a \le a$, whence $na \le a$ for any $n \in \N^*$. Since $\alpha$ is minimal, $a$ is an order unit in $T(\alpha)$ so $b \le a$ for any $b \in T(\alpha)$. This proves that the action has dynamical comparison.
\end{proof}

In the minimal case, one can provide yet another characterization of dynamical comparison.

\begin{defn}
Assume that $\alpha$ is minimal; we say that $T(\alpha)$ has the \emph{weak comparability property} if 
$$\forall a \ne 0 \ \exists k \in \N^* \  \forall b  \quad kb \le [X] \Rightarrow b \le a $$
\end{defn}

Note that our formulation above only applies to the case where $T(\alpha)$ is simple (i.e.~$\alpha$ is minimal). One can formulate weak comparability in non-simple semigroups but the definition is more involved; here we only use it for minimal $\alpha$ so we give the simpler reformulation above. One could add the condition that $a=[A]$ for some $A \in \Clopen(X)$ and obtain an equivalent definition (for a more general formulation of weak comparability, a detailed discussion and references, see section 1.6 of \cite{Wehrung2017}).

\begin{prop}\label{prop: link between weak comparability, almost unperforation and cancellation}
Assume that $\alpha$ is minimal. Then:
\begin{enumerate}
\item $T(\alpha)$ satisfies weak comparability if, and only if, it is almost unperforated.
\item If $T(\alpha)$ is stably finite (equivalently, if $\mcM(\alpha) \ne \emptyset$) then these conditions imply that $T(\alpha)$ is cancellative.
\end{enumerate}
\end{prop}

 This follows from Proposition 1.6.8 of \cite{Wehrung2017}, which contains a wealth of information on refinement monoids (the first statement follows from Theorem 4.1 of \cite{Ara_Goodearl_Pardo_Tyukavkin};
 the second one is a particular case of Theorem 1.7 in \cite{Ara_Pardo1996}). 

Recall that Downarowicz and Zhang \cite{Downarowicz_Zhang} proved recently that every Cantor action of a group of subexponential growth has the dynamical comparison property; Kerr and Naryshkin \cite{Kerr_Narishkyn} even more recently showed the same result for free actions of elementary amenable groups. It follows that free minimal Cantor actions of such amenable groups are all such that $T(\alpha)$ is cancellative. Interestingly, I do not know of a direct argument to prove cancellativity for these actions, and it seems to be an open problem whether clopen type semigroups of minimal Cantor actions of amenable groups are always cancellative.

\subsection{Tameness and dense locally finite groups}\label{s:tame}
In this section we assume that $X$ is metrizable. We recall that the group $\Homeo(X)$ is endowed with a natural Polish group topology, which comes from viewing homeomorphisms of $X$ as automorphisms of the Boolean algebra $\Clopen(X)$. A basis of neighborhoods of $1$ for this topology is given by subgroups of the form 
\[\{g \in \Homeo(X) \colon \forall A \in \mcA \ gA=A \} \]
where $\mcA$ ranges over all clopen partitions of $X$.

What we really care about in this section is the case of minimal Cantor actions of countable amenable groups. We recall some terminology.

\begin{defn}[Krieger \cite{Krieger1980}]
Given an algebra $\mcA$ of clopen subsets of $X$ and $G$ a subgroup of $\Homeo(X)$, we let $[G,\mcA]$ be the smallest subgroup of $\Homeo(X)$ containing $G$ and such that, for any partition $A_1,\ldots,A_n$ of $X$ with $A_i \in \mcA$, and $g_1,\ldots,g_n \in G$, if the mapping $h \colon x \in A_i \mapsto g_i(x)$ is an homeomorphism, then $h \in [G,\mcA]$.

We say that $G$ is a \emph{full group} if $G=[G,\Clopen(X)]$.

A subgroup $G$ of $\Homeo(X)$ is an \emph{ample group} if it is a countable, locally finite full group and $\{x \colon g(x)=x\}$ is clopen for every $g \in G$.
\end{defn}

\begin{defn}
The \emph{topological full group} $\llbracket \alpha \rrbracket$ of an action $\alpha \colon \Gamma \actson X$ is the smallest full group which contains $\alpha(\gamma)$ for every $\gamma \in \Gamma$. 
\end{defn}

Whenever $\alpha$ is a minimal $\Z$-action on a Cantor space $X$, there exists an ample group $\Lambda$ which is dense in $\llbracket \alpha \rrbracket$ (equivalently, $\Lambda$ and $\alpha$ have the same orbits on clopen sets). Any two such ample groups are conjugated, and they play an important role in the Giordano--Putnam--Skau classification of minimal $\Z$-actions up to orbit equivalence (see \cite{GPS_OE} or \cite{mellerayInvariantMeasuresOrbit2021} and the references given therein). 

Our aim in this section is to prove that existence of a dense locally finite subgroup of $\llbracket \alpha \rrbracket$ can be equivalently formulated as a property of $T(\alpha)$.

\begin{defn}
$T(\alpha)$ is \emph{unperforated} if for every $a,b \in T(\alpha)$ and every $n \in \N^*$ one has $na \le nb \Rightarrow a \le b$.
\end{defn}

%\begin{defn}[Ara--Goodearl \cite{Ara_Goodearl2015}]
%A refinement monoid is \emph{tame} if it is an inductive limit of finitely generated refinement monoids.
%\end{defn}

%Stably finite, tame conical monoids have been characterized as follows by Ara and Goodearl.
%
%\begin{theorem}[\cite{Ara_Goodearl2015}*{Theorem~3.14}]\label{t:tame_carac}
%A stably finite, conical monoid is tame iff it is unperforated and cancellative.
%\end{theorem}

\begin{lemma} \label{l:ample_is_unperforated}
Assume that $\alpha$ is an action on a compact, $0$-dimensional metrizable space $X$ such that $\llbracket \alpha \rrbracket$ has a dense, locally finite subgroup. Then $T(\alpha)$ is unperforated and cancellative.
\end{lemma}

\begin{proof}
Let $a,b \in T(\alpha)$ and $n \in \N$ be such that $na \le nb$.

Write $a= \sum_{i=1}^n [A_i]$, $b= \sum_{j=1}^m [B_j]$ where $A_i, B_j$ are clopen subsets of $X$. The assumption on $\llbracket \alpha \rrbracket$ implies that there exists a finite subgroup $\Lambda$ of $\llbracket \alpha \rrbracket$ and a finite $\Lambda$-invariant clopen partition $\mathcal A$ such that the relation $na \le nb$ is witnessed using elements of $\Lambda$ and $\mcA$. Letting $Z$ denote the finite set of atoms of $\mathcal A$, the type semigroup of the action $\Lambda \actson Z$ is unperforated (by a simple matching argument, see e.g. \cite{Wagon2016}*{Theorem~10.20}) and from this we see that $a\le b$ is witnessed by elements of $\Lambda$ and $\mcA$.

The proof of cancellativity is similar, since the type semigroup of an action of a finite group $\Lambda$ on a finite set $Z$ is cancellative by a simple counting argument: using the same notations as above $a=b$ means that, for each orbit $O$ of the action $\Lambda \actson Z$, there are as many elements of (a representative of) $a$ equivalent to some element of $O$ as there are such elements in $b$.
\end{proof}

Remarkably, the converse property also holds.

\begin{prop}\label{p:tame_gives_ample}
Let $\alpha$ be an action on a compact $0$-dimensional metrizable space $X$. Assume that $T(\alpha)$ is unperforated and cancellative. Then $\llbracket \alpha \rrbracket$ has a dense ample subgroup.
\end{prop}

Note that since $T(\alpha)$ is cancellative  we have for any $U,V \in \Clopen(X)$ that $[U]=[V]$ iff there exists $g \in \llbracket \alpha \rrbracket$ such that $gU=V$.

\begin{proof}
Given a finite Boolean subalgebra $\mcA$ of $\Clopen(X)$, we denote by $G_\mcA$ the subgroup of $\mathrm{Aut}(\mcA)$ made up of all automorphisms such that $[g(U)]=[U]$ for all $U \in \mcA$.

We claim that one can build inductively a sequence of finite Boolean subalgebras $\mcA_n$ of $\Clopen(X)$ such that:
\begin{itemize}
    \item For all $n$ $\mcA_n$ is a subalgebra of $\mcA_{n+1}$, and every element of $G_{\mcA_n}$ extends to an element of $G_{\mcA_{n+1}}$.
    \item For every $U,V \in \Clopen(X)$ such that $[U]=[V]$, there exists $n$ such that $U,V \in \mcA_n$ and $g \in G_{\mcA_n}$ such that $gU=V$.
\end{itemize}

To see why this construction can be carried out, fix a finite Boolean algebra $\mcA$. Fix also $U,V \in \Clopen(X)$ such that $[U]=[V]$. We need to build a finite Boolean algebra $\mcB$ refining $\mcA$, such that elements of $G_\mcA$ extend to elements of $G_\mcB$ and there exists $g \in G_\mcB$ such that $gU=V$.

Find a finite algebra $\mcA'$ refining $\mcA$, $U$, and $V$, then let $M$ be the finite subset of $T(\alpha)$ consisting of the types of atoms of $\mcA'$. Let $A, B$ be two atoms of $\mcA'$ such that $[A]=[B]$; denoting by $A_1,\ldots,A_n$ the atoms of $\mcA'$ contained in $A$, and $B_1,\ldots,B_m$ the atoms contained in $B$, we obtain the relation $\sum_{i=1}^n [A_i] = \sum_{j=1}^m [B_j]$ in $T(\alpha)$. Note that there are finitely many relations in $T(\alpha)$ that occur in such a way.

As pointed out in \cite{Ara_Goodearl2015}* {Theorem~3.14}, the assumptions on $T(\alpha)$ imply (for instance via the Effros--Handelman--Shen theorem) that it is the limit of an inductive sequence of free monoids $\N^{n_i}$. If we go far enough in this sequence, all equalities of types occuring as in the previous paragraph are witnessed; this means that there exists some integer $p$, as well as a map $\varphi \colon M \to \N^p$ and a morphism $\psi \colon \N^p \to T(\alpha)$ such that:
\begin{itemize}
\item $\psi \circ \varphi$ is the inclusion map from $M$ to $T(\alpha)$.
\item Whenever $A= \bigsqcup_{i=1}^n A_i$ and $B=\bigsqcup_{j=1}^m B_j$ are two elements of $\mcA'$ such that $[A]=[B]$ (with $A_i$, $B_j$ atoms of $\mcA'$) we have $\sum_{i=1}^n  \varphi([A_i])= \sum_{j=1}^m \varphi([B_j])$.
\end{itemize}
Let $e_1,\ldots,e_p$ denote the standard generators of $\N^p$. For each atom $A$ of $\mcA'$, we have a unique way to write $\varphi(A)= \sum_{i \in I_A} n_i e_i$. We can thus build a partition of $A$ made up of $n_i$ disjoint copies of $\psi(e_i)$ for all $i \in I_A$, i.e.~write 
\[A= \bigsqcup_{i \in I_A} \bigsqcup_{j=1}^{n_i} E_{i,j}(A) \]
with $[E_{i,j}(A)]=\psi(e_i)$ for all $i \in I_A$.

We thus obtain a clopen partition $\mcB$ of $X$, generated by (multiple) copies of $\psi(e_1),\ldots,\psi(e_n)$.

Assume that $A,B \in \mcA'$ are such that $[A]=[B]$. Then for each $i$ $A$ contains as many disjoint copies of $\psi(e_i)$ as $B$; by permuting these copies accordingly, we obtain that there exists an element $g$ of $G_\mcB$ such that $gA=B$. This proves both that elements of $G_{\mcA}$ extend to elements of $G_\mcB$, and that there exists $g \in G_\mcB$ such that $gU=V$.

This confirms that one can inductively perform the construction described at the beginning of the proof. By construction 
we have $\bigcup \mcA_n = \Clopen(X)$. Further, the inductive limit of the groups $G_{\mcA_n}$ induces an ample subgroup $\Lambda$ of $\Homeo(X)$ such that, for every $A,B \in \Clopen(X)$,
\[ \left( \exists \lambda \in \Lambda \ \lambda A = B \right) \Leftrightarrow \left( \exists g \in \llbracket \alpha \rrbracket \ g A = B \right)\]
We have not yet ensured that $\Lambda \subseteq \llbracket \alpha \rrbracket$; however, a variation on a construction of Krieger \cite{Krieger1980} (we postpone the proof to appendix \ref{a:Krieger}, since it is a mostly routine modification of Krieger's proof) implies that there exists $g \in \Homeo(X)$ such that $g \overline{\llbracket \alpha \rrbracket} g^{-1}= \overline{\llbracket \alpha \rrbracket}$ and $g \Lambda g^{-1} \subset \llbracket \alpha \rrbracket$. Then $g \Lambda g^{-1}$ is the desired dense ample subgroup of $\llbracket \alpha \rrbracket$.
\end{proof}

The next theorem sums up the previous two results.

\begin{theorem}\label{p:unperf_locally_finite_dense}
Assume that $\alpha$ is an action of a countable group $\Gamma$ on a compact, metrizable, $0$-dimensional space. Then $\llbracket \alpha \rrbracket$ admits a dense locally finite subgroup (even, a dense ample subgroup) iff $T(\alpha)$ is unperforated and cancellative.
\end{theorem}

If $\Gamma$ is amenable and $\alpha$ is minimal, we already mentioned that by a result of Ara and Pardo $T(\alpha)$ is cancellative as soon as it is stably finite and almost unperforated (since it then has the weak comparability property), so the assumption of cancellativity is redundant in that case.

%\begin{proof}
%The existence of a dense locally finite subgroup in $\llbracket \alpha \rrbracket$ gives us unperforation and cancellativity by Lemma \ref{l:ample_is_unperforated}. Conversely, \cite{Ara_Goodearl2015}* {Theorem~3.14}  gives us that unperforation, cancellativity and conicality imply tameness (as well as stable finiteness, of course) and Proposition \ref{p:tame_gives_ample} gives the desired implication.
%\end{proof}

Matui \cite{Matui_torsion} gave examples of free minimal Cantor $\Z^2$-actions for which the torsion part of the $0$-homology group $H(\alpha)$ is nontrivial (this contradicts a result of Forrest--Hunton \cite{Forrest_Hunton} whose proof had a gap). For these actions we have $T(\alpha)=H(\alpha)^+$ since $T(\alpha)$ is cancellative (see e.g. Proposition \ref{prop: link between weak comparability, almost unperforation and cancellation})
so torsion in $H(\alpha)$ means that $T(\alpha)$ is not unperforated. We thus obtain the following result.

\begin{theorem}\label{t:no_dense_ample_group}
There exists a free minimal Cantor action $\alpha$ of $\Z^2$ on the Cantor space such that $\llbracket \alpha \rrbracket$ does not have any dense, locally finite subgroup.
\end{theorem}

\section{Dynamical Comparison in the Stone-\v{C}ech compactification and in the universal minimal flow}\label{s:Stone_Cech}
We again let $\Gamma$ be a countable discrete group, and denote by $\beta \Gamma$ its Stone-\v{C}ech compactification; it is a compact, Hausdorff, $0$-dimensional space (and very much non-metrizable unless $\Gamma$ is finite). 

The clopen subsets of $\beta \Gamma$ correspond to subsets of $\Gamma$ (identifying $\beta \Gamma$ with the set of all ultrafilters on $\Gamma$, clopen subsets of $\beta \Gamma$ are of the form $\{\mathcal U \colon A \in \mathcal U\}$ for $A \subseteq \Gamma$). Thus the action of $\Gamma$ on $\beta \Gamma$ corresponds to the well-studied action of $\Gamma$ on its subsets via translation $\gamma \cdot A = \{\gamma x \colon x \in A\}$, and equidecomposability of clopen subsets in $\beta \Gamma$ is the same as classical equidecomposability of subsets in $\Gamma$.

One should note here that, for instance, $T(\Z \actson \beta \Z)$ is not stably finite, since $[\{0\}] + [\beta \Z] = [\beta \Z]$ (as witnessed by extending to $\beta \Z$ the map sending each $n \ge 0$ to $n+1$, and each $n < 0$ to $n$).
Actually, an inequality of the form $2a \le a$, with $a$ nonzero, is possible in $T(\Gamma \actson \beta \Gamma)$ even for a countable amenable $\Gamma$. Indeed, there exists such an $a$ iff $\Gamma$ is not supramenable; and there exist some countable, solvable, not supramenable groups. This kind of example is not restricted to non-metrizable spaces: using arguments similar to those in the proof of Proposition \ref{p:metrizable_factor} below, it follows that there exist free topologically transitive Cantor actions of some solvable groups with a nonempty clopen set $A$ such that $2[A] \le [A]$. 

Still, classical results show that type semigroups of Stone-\v{C}ech compactifications are rather well-behaved.

\begin{theorem}[\cite{Wagon2016}*{Theorems~3.6 and~10.20}]
The semigroup $T(\Gamma \actson \beta \Gamma)$ is unperforated and $\le$ is a partial order on $T(\Gamma \actson \beta \Gamma)$.
\end{theorem}

Note that when $\Gamma$ is not amenable there is no invariant Radon probability measure on $\beta \Gamma$ and Lemma \ref{l:noinvmeasure} then implies that this action does not have dynamical comparison.

Since clopen sets in $\beta \Gamma$ correspond to subsets of $\Gamma$, states on $T(\Gamma \actson \beta \Gamma)$ correspond to finitely additive $\Gamma$-invariant probability measures on $\Gamma$. 
The following result is an immediate consequence of this, along with our earlier observations. 
\begin{theorem}\label{t:Tarski}
Let $\Gamma$ be a countable amenable group. Then $\Gamma \actson \beta \Gamma$ has dynamical comparison. 

This implies the following fact: let $A,B$ be two subsets of $\Gamma$ such that $\mu(A)< \mu(B)$ for any $\Gamma$-invariant f.a.p.m on $\Gamma$. Then there exist $A_1,\ldots,A_n \subset \Gamma$ and $\gamma_1,\ldots,\gamma_n \in \Gamma$ such that $A=\bigsqcup_{i=1}^n A_i$, $\bigsqcup_{i=1}^n \gamma_i A_i \subset B$.
\end{theorem}

\begin{proof}
We know that $T(\Gamma \actson \beta \Gamma)$ is unperforated, whence it has dynamical comparison since $\Gamma$ is amenable (see Lemma \ref{l:dyn_comp_unperf_below_order_unit}). 

The second part of the statement is just a reformulation of dynamical comparison without the semigroup terminology, since $\Gamma$-invariant finitely additive probability measures on $\Gamma$ correspond to $\Gamma$-invariant Radon measures on $\beta \Gamma$.
\end{proof}

% (q: can one use comparison to say something interesting for nonamenable groups? It says that $A$ is equidecomposable with $\Gamma$ iff finitely many translates of $A$ cover $\Gamma$, but that is obvious from cancellation + paradoxicality). 

The previous result can also be derived from the finitileability theorem of Downarowicz, Huczek and Zhang \cite{Downarowicz_Huczek_Zhang2015} via a standard compactness argument. The proof presented here is much more elementary.
% Maybe it is worth noting here that, by an argument similar to \cite{Kerr_Szabo2020}*{Proposition 3.3}, for any $A \subset \Gamma$ the set of all possible measures $\mu(A)$, for $\mu$ a $\Gamma$-invariant probability measure on $\Gamma$, is equal to $[\underline{D}(A),\overline{D}(A)]$, where $\underline{D}(A)$ and $\overline{D}(A)$ stand for the lower and upper Banach densities of $A$. These quantities feature prominently in \cite{Kerr_Szabo2020} and \cite{Downarowicz_Huczek_Zhang2015}. It is not clear whether one can use unperforation of $\beta \Gamma$ to deduce results about the existence of F\o lner tilings; the relation between comparison and tilings should probably be explored further (in the metrizable case see REFS).

It seems worth noting that, as explained to the author by B. Weiss, there exists $A \subset \Z$ such that $\mu(A)=\frac{1}{2}$ for every $\mu \in \mcM(\Z \actson \beta \Z)$, yet $A$ and $\Z \setminus A$ are not equidecomposable. To see this, write $\Z$ as a disjoint union of arithmetic sequences, e.g.
\[\Z=\{2n \colon n \in \Z\}\sqcup\{4n+1 \colon n \in \Z\}  \sqcup \{8n+3 \colon n\in \Z\} \ldots\] Then let $A$ consist of every other element from each of these sequences, i.e.~
\[A=\{4n \colon n \in \Z\} \sqcup\{8n + 1 \colon n \in \Z\} \sqcup \{16n+3 \colon n \in \Z\} \ldots\]
Clearly for every $\mu \in \mcM(\Z \actson \beta \Z)$ we have $\mu(A) \ge \frac{1}{4} + \frac{1}{8} + \frac{1}{16} + \ldots = \frac{1}{2}$; the same holds for $B=\N \setminus A$ and it follows that $\mu(B)=\mu(A)=\frac{1}{2}$ for every invariant $\mu$. Yet one can check that infinitely many shifts are needed to map $A$ into $B$, so $A$ and $B$ are not equidecomposable.

Next, we consider the universal minimal flow of $\Gamma$, which we denote $\Gamma \actson \mu \Gamma$. We recall that $\mu \Gamma$ is a compact, Hausdorff, $0$-dimensional space and that $\Gamma \actson \mu \Gamma$ is free. Further, any minimal subset for the action $\Gamma \actson \beta \Gamma$ is isomorphic to $\mu \Gamma$; we identify $\mu \Gamma$ with some fixed minimal subset of $\beta \Gamma$, and choose a $\Gamma$-equivariant retraction $r \colon \beta \Gamma \to \mu \Gamma$. We refer the reader to \cite{Pestov} and its bibliography for more details on this area of topological dynamics. 

The retraction $r$ induces a homomorphism 
\[ \varphi \colon (T(\Gamma \actson \mu \Gamma),+) \to (T(\Gamma \actson \beta \Gamma),+) \] defined by setting $\varphi([A]) =[r^{-1}(A)]$ for $A \in \Clopen(\mu \Gamma)$, and then extending to $T(\Gamma \actson \mu \Gamma)$.

In the other direction, one can define $\psi \colon T(\Gamma \actson \beta \Gamma) \to T(\Gamma \actson \mu \Gamma)$ by setting $\psi([A]) =[A \cap \mu \Gamma]$ for $A \in \Clopen (\beta \Gamma)$, and extending to $T(\Gamma \actson \beta \Gamma)$ .

Since $r$ is a retraction we have $\psi (\varphi(a))=a$ for all $a \in T(\Gamma \actson \mu \Gamma)$. Hence $T(\Gamma \actson \mu \Gamma)$ is isomorphic to a subsemigroup of $T(\Gamma \actson \beta \Gamma)$.

The following result is an immediate consequence of this and Theorem \ref{t:Tarski}.

\begin{theorem}\label{t:minimal_flow_comparison} Let $\Gamma$ be a countable group. The semigroup
$T(\Gamma \actson \mu \Gamma)$ is unperforated, and $\le$ is a partial order on it. 

In particular, the action $\Gamma \actson \mu \Gamma$ has dynamical comparison.

\end{theorem}
(Recall that by Proposition \ref{p:minimal_dynamical_unperforated} if $\alpha$ is minimal and $T(\alpha)$ is almost unperforated then $\alpha$ has the dynamical comparison property)
 
If $\Gamma$ is not amenable, this (along with $\le$ being a partial order) implies that $a=b$ for every nonzero $a,b \in T(\mu \Gamma)$. In particular any two nonempty clopen subsets of $\mu \Gamma$ are equidecomposable.

In the amenable case, dynamical comparison implies that $T(\Gamma \actson\mu \Gamma)$ is cancellative.
Amusingly, we seem to have found a new condition to add to the long list of characterizations of amenability for countable groups: $\Gamma$ is amenable iff $T(\Gamma \actson\mu \Gamma)$ is stably finite iff $T(\Gamma \actson\mu \Gamma)$ is cancellative.

The following proposition is a routine consequence of the aforementioned properties of $\mu \Gamma$.

\begin{prop}\label{p:metrizable_factor}
Any minimal Cantor action of $\Gamma$ is a factor of a free, minimal Cantor action whose clopen type semigroup is unperforated and partially ordered (in particular, this action has the dynamical comparison property and if $\Gamma$ is amenable its clopen type semigroup is cancellative).
\end{prop}

\begin{proof}
Given a minimal Cantor action $\alpha \colon \Gamma \actson X$, find a factor map $\pi \colon \mu\Gamma \to X$. 

Let $\mcA = \pi^{-1}(\Clopen(X))$, which is a countable, $\Gamma$-invariant Boolean subalgebra of $\Clopen(\mu \Gamma)$. 

Given a $\Gamma$-invariant Boolean subalgebra $\mcB$ of $\Clopen(\mu \Gamma)$, we denote by $T_\mcB(\Gamma \actson \mu \Gamma)$ the semigroup of types obtained by considering only decompositions with pieces in $\mcB$ (still with $\Gamma$ as the acting group), and $\le_\mcB$ the corresponding partial ordering. We also denote equality of types for this relation by $=_\mcB$.

We build a sequence of countable atomless subalgebras $(\mcA_n)_n$ of $\Clopen(X)$ as follows. First, we let $\mcA_0$ be any countable, $\Gamma$-invariant, atomless Boolean subalgebra of $\Clopen(\mu \Gamma)$ containing $\mcA$ and such that, for any $\gamma \in \Gamma \setminus \{1\}$, there exists a clopen partition $(U_i)_{i \in I}$ of $\mu \Gamma$ whose elements belong to $\mcA_0$ and are such that $\gamma U_i \cap U_i = \emptyset$ for all $i \in I$ (if necessary, see the proof of \ref{p:free_G_delta} below for an explanation of why such clopen partitions exist).

Next, assume that $\mcA_n$ has been built; for any $u,v \in T_{\mcA_n}(\Gamma \actson \mu \Gamma)$ and any $k \in \N^*$ such that $k u \le_{\mcA_n} k v$, find clopen sets forming decompositions of $u,v$ witnessing that $u=v$ in $T(\Gamma \actson \mu \Gamma)$; and similarly if $u \le_{\mcA_n} v \wedge v \le _{\mcA_n} u$ find clopen decompositions witnessing that $u=v$. Take any $\Gamma$-invariant, countable, atomless Boolean subalgebra $\mcA_{n+1}$ that contains $\mcA_n$ and all these new clopen sets and denote by $\pi_n$ the natural map from $T_{\mcA_n}(\Gamma \actson \mu \Gamma)$ to $T_{\mcA_{n+1}}(\Gamma \actson \mu \Gamma)$ (if two clopen sets are equidecomposable with pieces in $\mcA_n$ then they are also equidecomposable with pieces in $\mcA_{n+1}$). We have:

\begin{itemize}
    \item $ \forall u,v \in T_{\mcA_n}(\Gamma \actson \mu \Gamma) \ \forall k \in \N^* \quad k u \le_{\mcA_n} kv \Rightarrow \pi_n(u) \le_{\mcA_{n+1}} \pi_n(v)$. 
    \item $\forall u,v \in T_{\mcA_n}(\Gamma \actson \mu \Gamma) \quad (u \le_{\mcA_n}  v \wedge v \le_{\mcA_n} u) \Rightarrow \pi_n(u)=_{\mcA_{n+1}} \pi_n(v)$. 
\end{itemize}

Then $\mcB = \bigcup_n \mcA_n$ is countable, atomless, $\Gamma$-invariant, contains $\mcA$, and satisfies the following conditions:
\begin{itemize}
    \item For any $u,v \in T_\mcB(\Gamma \actson \mu \Gamma)$, and any $n \in \N^*$, $n u \le_\mcB nv \Rightarrow u \le_{\mcB} v$. 
    \item For any $u,v \in T_\mcB(\Gamma \actson \mu \Gamma)$, $(u \le_\mcB  v \wedge v \le_\mcB u) \Rightarrow u=_\mcB v$. 
    \item For any $\gamma \in \Gamma \setminus \{1\}$, there exists a clopen partition $(U_i)_{i \in I}$ of $\mu \Gamma$ whose elements belong to $\mcB$ and are such that $\gamma U_i \cap U_i = \emptyset$ for all $i \in I$.
\end{itemize}
Denote by $Y$ the compactification of $\Gamma$ associated with $\mcB$, i.e.~ the compactification whose algebra of continuous functions is the closed subalgebra of $C(\mu \Gamma)$ whose $\{0,1\}$-valued functions are the indicator functions of elements of $\mcB$. 
Then $Y$ is a Cantor space, and we have a free action $\Gamma \actson Y$ whose clopen type semigroup is unperforated and partially ordered. Since $\Clopen(\mu \Gamma) \supset \mcB \supseteq \mcA$, $\pi$ factors through a $\Gamma$-equivariant map $\psi \colon Y \to X$, and we are done.
\end{proof}

% QUESTION: is $T(\mu \Z)$ saturated?

\section{Spaces of actions}
We fix a Cantor space $X$, an infinite countable group $\Gamma$ and consider the space $A(\Gamma)$ of actions of $\Gamma$ on $X$, which we see as a closed subset of $\Homeo(X)^\Gamma$, hence a Polish space.

\subsection{\texorpdfstring{Some $G_\delta$ subsets of $A(\Gamma)$}{Some Gdelta subsets of A(Gamma)}}

The following is well-known.
\begin{prop}
The space $\Min(\Gamma)$ of all minimal Cantor actions of $\Gamma$ is a $G_\delta$ subset of $A(\Gamma)$.
\end{prop}

\begin{proof}
An action $\alpha$ is minimal iff for every nonempty clopen $A$ there exist $\gamma_1,\ldots,\gamma_n$ such that $\bigcup_i \alpha(\gamma_i) A=X$. For fixed $A$, this is an open condition on $\alpha$.
\end{proof}

Hence $\Min(\Gamma)$ is a Polish space in its own right.

\begin{prop}\label{p:free_G_delta}
The space $\Free(\Gamma)$ of all free actions is a $G_\delta$ subset of $A(\Gamma)$.
\end{prop}

\begin{proof}
Assume that $\alpha$ is a free action, and fix $\gamma \in \Gamma \setminus \{1\}$. Then for every $x \in X$ there exists a clopen $U$ containing $x$ and such that $\alpha(\gamma) U \cap U = \emptyset$. Hence one can find finitely many clopen subsets $A_1,\ldots,A_n$ such that $\bigcup_i A_i=X$ and $\alpha(\gamma) A_i \cap A_i = \emptyset$ for all $i$. 

Conversely, any $\alpha$ satisfying the condition in the previous paragraph is free. Thus 
\[\Free(\Gamma)= \bigcap_{\gamma \in \Gamma \setminus\{1\}} \bigcup_{\mcA} \{\alpha \colon \forall A \in \mcA \ \alpha(\gamma) A \cap A=\emptyset \}\]
(in the line above $\mcA$ runs over the set of all clopen partitions of $X$)
\end{proof}
Note that Proposition \ref{p:metrizable_factor} shows in particular that for any $\Gamma$ there exists a free minimal action of $\Gamma$ on the Cantor space, a well-known fact (see e.g.~\cite{Hjorth_Mollberg}).
We are particularly interested in generic properties (in the sense of Baire category) in $\Min(\Gamma)$.

\begin{lemma} \label{l:cluster_invariant}
Assume that $(\alpha_n) \in A(\Gamma)^{\N}$ converge to some $\alpha \in A(\Gamma)$, and $\mu_n \in \mcM(\alpha_n)$ for all $n$. Let $\mu$ be a cluster point of $(\mu_n)_n$. Then $\mu \in \mcM(\alpha)$.
\end{lemma}

\begin{proof}
We might as well assume that $(\mu_n)_n$ converges to $\mu$. Fix $A \in \Clopen(X)$ and $\gamma \in \Gamma$. There exists $N$ such that for any $n \ge N$ one has $\alpha_n(\gamma)A= \alpha(\gamma)A$. Hence $$\forall n \ge N \quad \mu_n(A)= \mu_n(\alpha_n(\gamma)A)= \mu_n(\alpha(\gamma)A)$$
Letting $n$ go to $\infty$, we obtain $\mu(A)= \mu(\alpha(\gamma)A)$. 
\end{proof}

\begin{prop}
The set $A^*(\Gamma)$ of all actions which admit an invariant probability measure is closed in $A(\Gamma)$.

The set $A^*_1$ of all uniquely ergodic actions of $\Gamma$ is a $G_\delta$ subset of $A(\Gamma)$.
\end{prop}

\begin{proof}
The first fact is an immediate consequence of the previous lemma allied to the compactness of the space of Borel probability measures on $X$.

To see why the second fact holds, fix a clopen $A$ and $\varepsilon >0$. Then it follows from the previous lemma that 
$$\Omega_{A,\varepsilon}= \{\alpha \colon \exists \mu_1,\mu_2 \in \mcM(\alpha) \ |\mu_1(A)-\mu_2(A)| \ge \varepsilon\}$$
is closed in $A(\Gamma)$. Hence 
$$\Sigma_{A}=\{\alpha \colon \forall \mu_1,\mu_2 \in \mcM(\alpha) \ \mu_1(A)=\mu_2(A)\}$$
is $G_\delta$ in $A(\Gamma)$. Considering the intersection of all $\Sigma(A)$ as $A$ ranges over clopen subsets of $X$, we conclude.
\end{proof}

\begin{prop} The following subsets of $A(\Gamma)$ are all $G_\delta$:
\begin{itemize}
    \item Actions with the dynamical comparison property.
    \item Actions such that $T(\alpha)$ is almost unperforated.
    \item Actions such that $T(\alpha)$ is unperforated.
    \item Actions such that $T(\alpha)$ is cancellative.
    \item Actions such that $\le$ is a partial order on $T(\alpha)$.
%    \item Saturated actions.
\end{itemize}

\end{prop}

\begin{proof}
Let us only prove the first fact, the others being similar. Fix $A,B \in \Clopen(X)$. By Lemma \ref{l:cluster_invariant} the set $\Sigma(A,B)=\{\alpha \colon \exists \mu \in \mcM(\alpha) \ \mu(A) \ge \mu(B)\}$ is closed in $A(\Gamma)$. The set $\Delta(A,B)=\{\alpha \colon [A] \le [B] \text{ in } T(\alpha)\}$ is open in $A(\Gamma)$. The set of all actions of $\Gamma$ with the dynamical comparison property is equal to 
\[\bigcap_{A,B \in \Clopen(X)} \big(\Sigma(A,B) \cup \Delta(A,B) \big) \]
hence it is a $G_\delta$ subset of $A(\Gamma)$ (since the union of a closed and an open set is $G_\delta$, and there are only countably many clopen subsets in $X$).
\end{proof}

% \subsection{Weak containment}

% Characterization via closure of the conjugacy class. Compactness of the space of weak equivalence classes.

\subsection{\texorpdfstring{Conjugacy classes in $\Min(\Gamma)$}{Conjugacy classes in Min(Gamma)}}

Below we will repeatedly make use of the well-known and easily-verified fact that if $\alpha,\beta$ are two Cantor actions and $\alpha$ is a factor of $\beta$, then $\alpha$ belongs to the closure of the conjugacy class of $\beta$.

\begin{prop}
There exist actions with dense conjugacy classes in $A(\Gamma)$ and $\Min(\Gamma)$.
\end{prop}

\begin{proof}
The case of $A(\Gamma)$ is well known: let $(\alpha_n)_n$ enumerate a dense subset of $A(\Gamma)$, then consider the product action $\alpha$. It is again a Cantor action, and each $\alpha_n$ is a factor of $\alpha$. In particular each $\alpha_n$ is contained in the closure of the conjugacy class of $\alpha$, which is thus dense in $A(\Gamma)$.

The case of $\Min(\Gamma)$ is similar; enumerate a dense subset $(\alpha_n)_n$ of $\Min(\Gamma)$ and let $\alpha$ be any minimal component of the product action. Then $\alpha$ is a minimal Cantor action which maps into every $\alpha_n$, so by minimality each $\alpha_n$ is a factor of $\alpha$.
\end{proof}

The existence of dense conjugacy classes implies that every Baire-measurable, conjugacy invariant subset of $\Min(\Gamma)$ is either meager or comeager (the so-called $0$-$1$ topological law, see \cite{Kechris1995}).

\subsection{\texorpdfstring{Generic properties in $\Min(\Gamma)$}{Generic properties in Min(Gamma)}}

For $\Z$, there exists a generic conjugacy class (the universal odometer \cite{Hochman2008}). What about other groups? I do not know the answer to this question even for $\Z^2$.

Proposition \ref{p:metrizable_factor} immediately implies the following:

\begin{prop}\label{p:generic_properties}
A generic element $\alpha$ of $\Min(\Gamma)$ is such that:
\begin{enumerate}
    \item $\alpha$ is free.
    \item $T(\alpha)$ is unperforated (hence $T(\alpha)$ has dynamical comparison).
    \item The algebraic order on $T(\alpha)$ is a partial ordering.
\end{enumerate}
\end{prop}

The fact that a generic minimal Cantor action of an amenable group has the dynamical comparison property also follows from \cite{Conley_etc2018} and the relation between \emph{almost finiteness} and dynamical comparison (see Theorem 6.1 of \cite{Kerr_Szabo2020}). Conversely using the previous proposition one recovers Theorem 4.2 of \cite{Conley_etc2018} by applying Theorem 6.1 of \cite{Kerr_Szabo2020}.

\section{Weakenings of dynamical comparison} \label{s:measured_comparison}
This work was in part motivated by the following question: given a minimal action $\alpha$ of a countable amenable group on the Cantor space, does there exist a minimal $\Z$-action $\beta$ such that $\mcM(\alpha)=\mcM(\beta)$? Of course, if $\alpha$ is orbit equivalent to a $\Z$-action then the previous question has an affirmative answer.

The criterion established in \cite{Melleray_dynamicalsimplices} (which slightly simplifies a condition obtained in \cite{Melleray2017}) shows that there exists such an action $\beta$ as soon as $\mcM(\alpha)$ is such that, for any $A,B \in \Clopen(X)$ such that $\mu(A) < \mu(B)$ for every $\mu \in \mcM(\alpha)$, there exists $C \in \Clopen(X)$ such that $\mu(A)=\mu(C)$ for all $\mu \in \mcM(\alpha)$ and $C$ is contained in $B$. This is a weakening of dynamical comparison.

\begin{defn}
Let $\alpha$ be a minimal action of $\Gamma$ on a Cantor space $X$. Denote by $K(\alpha)$ the group $\{g \in \Homeo(X) \colon \forall \mu \in \mcM(\alpha) \ g_*\mu=\mu\}$.

We say that $\alpha$ has the \emph{measured comparison property} if the action $K(\alpha) \actson X$ has dynamical comparison.
\end{defn}

Here we are departing from our earlier conventions, since $K(\alpha)$ is not countable, and we only defined comparison for countable groups (this played a part in some arguments, when we are using an enumeration of $\Gamma$). However, whenever $X$ is metrizable, $\Homeo(X)$ is separable, whence $K(\alpha)$ admits a countable dense subgroup $\Lambda$. Then $\mcM(K(\alpha))= \mcM(\Lambda \actson X)$ ($=\mcM(\alpha)$), $T(K(\alpha))=T(\Lambda \actson X)$ so all our earlier results apply equally well to any subgroup of $\Homeo(X)$.

\begin{lemma}
Let $\alpha$ be a minimal action of an amenable group on the Cantor space $X$. Then there exists a $\Z$-action $\beta$ such that $\mcM(\alpha)=\mcM(\beta)$ iff $\alpha$ has measured comparison.
\end{lemma}

\begin{proof}
Implication from left to right is immediate, since $\Z$-actions have dynamical comparison, which implies measured comparison.

Conversely, assume that $\alpha$ has measured comparison and let $A,B \in \Clopen(X)$ be such that $\mu(A) < \mu(B)$ for all $\mu \in \mcM(\alpha)$.
Since $\mcM(\alpha)=\mcM(K(\alpha))$, there exists $g \in \llbracket K(\alpha) \rrbracket=K(\alpha)$ such that $gA \subset B$, so the criterion of \cite{Melleray_dynamicalsimplices} is satisfied.
\end{proof}

So the various characterizations obtained in section \ref{s:clopen_type_semigroup} have counterparts for measured comparison. Unfortunately, this does not seem to lead to new examples of minimal actions with measured comparison, because it seems difficult to say anything meaningful about the group $K(\alpha)$. The same is true of another, intermediate weakening of dynamical comparison which we briefly discuss now.

\begin{defn}
Let $\alpha \colon \Gamma \actson X$ be an action of a countable group $\Gamma$ on the Cantor space. The \emph{full group} of $\alpha$ is the group 
\[ [\alpha] = \{g \in \Homeo(X) \colon \forall x \in X \ \exists \gamma \in \Gamma \ g(x)= \alpha(\gamma)x\}\]
\end{defn}

\begin{defn}
The action $\alpha$ has \emph{orbital comparison} if $[\alpha] \actson X$ has dynamical comparison.
\end{defn}

Since $\mcM(\alpha)=\mcM([\alpha])=\mcM(K(\alpha))$, dynamical comparison implies orbital comparison, which implies measured comparison. None of the converse implications is clear. If an action $\alpha$ is orbit equivalent to a $\Z$-action, then it must satisfy orbital comparison.

By the same argument as in the proof of Lemma~3.5 in \cite{GlasnerWeiss}, orbital comparison does imply that $\overline{[\alpha]}= K(\alpha)$. Even in the case where there are no invariant measures, I do not know of any example where this equality does not hold (it follows from dynamical comparison, and I do not know examples where dynamical comparison fails for minimal Cantor actions).

\begin{question}
Let $\alpha \colon \Gamma \actson X$ be a minimal action on the Cantor space. Is the full group $[\alpha]$ dense in $K(\alpha)=\{g \in \Homeo(X) \colon \forall \mu \in \mcM(\alpha) \ g_*\mu=\mu\}$?
\end{question}

Note that, if $\mcM(\alpha)$ is empty, one has $K(\alpha)=\Homeo(X)$, and it seems unlikely that emptiness of $\mcM(\alpha)$ is enough to guarantee that $[\alpha]$ is dense in $\Homeo(X)$ (though there certainly are examples where $\llbracket \alpha \rrbracket$ is not dense, like the action of a nonabelian free group on its boundary).

\section{Concluding remarks and questions}

We collect a few questions that came up in the paper; we mostly focus on the case of minimal actions.

\subsection{On the clopen type semigroup}
Even though dynamical comparison has so far mostly been studied for actions of amenable groups, the notion makes sense in general and I do not even know the answer to the following question.

\begin{question}\label{q:everyone_has_dynamical_comparison}
Does any minimal action of a countable group on the Cantor space satisfy the dynamical comparison property?
\end{question}

We noted early on in the paper that cancellativity of $T(\alpha)$ is an important condition, which is not always satisfied even by free, topologically transitive $\Z$-actions. However, in the minimal case, when the action admits an invariant measure and has dynamical comparison then $T(\alpha)$ is cancellative (see Proposition \ref{prop: link between weak comparability, almost unperforation and cancellation}). The argument for that is rather indirect as it goes through the fact that dynamical comparison is equivalent to weak comparability in this context, and a simple, conical, stably finite refinement monoid with weak comparability is cancellative by a theorem of Ara and Pardo.

\begin{question}\label{q:everyone_is_cancellative}
Is it true that $T(\alpha)$ is cancellative for any minimal Cantor action of a countable amenable group?
\end{question}

By Corollary 1.9 of \cite{Ara_Pardo1996}, to obtain a positive answer it would be enough to prove that $T(\alpha)$ is \emph{strictly cancellative}, i.e. that $a+b < a+c \Rightarrow b < c$.

In the case of $\Z$, cancellativity for minimal actions is easy to establish thanks to the existence of a dense locally finite subgroup of $\llbracket \alpha \rrbracket$. Even for $\Z^2$ I do not know a simple proof of cancellativity for minimal actions.

\begin{question}\label{q:everyone_equidecomposition_given_by_top_full_group}
Let $\alpha$ be a minimal action of a countable group on the Cantor space, and let $A$, $B$ be two clopens such that $[A]=[B]$ in $T(\alpha)$. Must there exist $g \in \llbracket \alpha \rrbracket$ such that $gA=B$?
\end{question}

For amenable groups, a positive answer to Question \ref{q:everyone_has_dynamical_comparison} implies a positive answer to Question \ref{q:everyone_is_cancellative}, and a positive answer to Question \ref{q:everyone_is_cancellative} implies a positive answer to Question \ref{q:everyone_equidecomposition_given_by_top_full_group}.

\subsection{On the space of actions}
We fix a countable group $\Gamma$ and denote again by $\Min(\Gamma)$ the set of all minimal actions of $\Gamma$ on the Cantor space $X$, with its usual Polish topology.

\begin{question}
For which groups $\Gamma$ is there a generic conjugacy class in $\Min(\Gamma)$?
\end{question}

It seems that, in order to make progress on this question, it would be useful to determine the closure of $\Min(\Gamma)$ inside the space of actions $A(\Gamma)$. When $\Gamma=\Z$ this is understood: by Theorem 5.9 of \cite{Bezugly_Dooley_Kwiatkowski2004} the closure of the set of minimal homeomorphisms in $\Homeo(X)$ is the set of all homeomorphisms $g$ such that for any nontrivial clopen set $U$ both $gU \setminus U$ and $U \setminus gU$ are nonempty. 

\begin{question}
What is the closure of $\Min(\Gamma)$ inside $A(\Gamma)$? Can one give a description similar to the case $\Gamma=\Z$?
\end{question}

It seems very unlikely that such a description is possible in general.

\begin{question}
Let $\Gamma$ be a countable, amenable group. Is unique ergodicity generic in $\Min(\Gamma)$? More generally, can one say anything meaningful about the set of invariant Borel probability measures of a generic minimal action?
\end{question}

We at least know that the set of all uniquely ergodic actions is $G_\delta$, and is nonempty when $\Gamma$ is amenable. To see that it is nonempty, one can use a generalization of the Jewett--Krieger theorem to amenable groups, due to Rosenthal; or one can apply a more general theorem of Frej--Huczek \cite{Huczek_Frej2018}.

In the non-amenable case, it seems that we know even less.

\begin{question}
Does every countable group admit a strictly ergodic action on the Cantor space?
\end{question}

Elek \cite{Elek2021} recently proved that every countable group admits a free, minimal action on the Cantor space which preserves a Borel probability measure.

\appendix

\section{A variation on an argument of Krieger}\label{a:Krieger}
We now describe an argument that completes the proof of Proposition \ref{p:tame_gives_ample}.
We fix a metrizable $0$-dimensional compact space $X$.

\begin{defn}[Krieger \cite{Krieger1980}]
Let $G$ be a subgroup of $\Homeo(X)$.

Given $g \in \Homeo(X)$ and a subalgebra $\mcA \subseteq \mcB_X$ which is $g$-invariant, we denote by $g_{|\mcA}$ the automorphism of $\mcA$ induced by $g$, and by $G_{|\mcA}= \{g_{|\mcA} \colon g \in G\}$. We say that $(\mcA,G)$ is a \emph{unit system} if:
\begin{itemize}
\item $\mcA$ is a subalgebra of $\Clopen(X)$.
\item $G$ leaves $\mcA$ invariant and $G=[G,\mcA]$.
\item The mapping $g \mapsto g_{|\mcA}$ is an isomorphism from $G$ to $\mathrm{Aut}(\mcA)$.
\end{itemize}
(note that it follows from these conditions that $\{x \in X \colon g(x)=x\}$ belongs to $\mcA$)
\end{defn}

\begin{defn}
We say that a unit system $(\mcA,\Gamma)$ is \emph{compatible} with a full group $G$ if for any $A \in \mcA$ and any $\gamma \in \Gamma$ there exists $g \in G$ such that $\gamma A = g A$.
\end{defn}

The main tool for our construction is an adaptation to our context of Lemma 3.4 of \cite{Krieger1980}, and the proof is essentially the same as Krieger's; we describe the argument for the reader's convenience (our formulation here is quite close to the one used for another variation on Krieger's theorem given in \cite{mellerayInvariantMeasuresOrbit2021}). 

\begin{lemma}\label{l:Krieger}
Fix a full group $G \subset \Homeo(X)$. Assume that we are given finite unit systems $(\mcA,\Delta)$, $(\mcC,\Lambda)$  compatible with $G$, along with an isomorphism $\Phi \colon \mcA \to \mcC$ with the following properties:
\begin{enumerate}
\item For all $A \in \mcA$, there exists $g \in G$ such that $\Phi(A)=g(A)$.
\item $\Lambda_{|\mcC}= \Phi \Delta_{|\mcA} \Phi^{-1}$.
\end{enumerate}
Then, for every $G$-compatible finite unit system $(\mcA',\Delta')$ refining $(\mcA,\Delta)$, there exists a $G$-compatible finite unit system $(\mcC',\Lambda')$ and an isomorphism $\Psi \colon \mcA' \to \mcC'$ that extends $\Phi$ and such that the conditions above are still satisfied. 
\end{lemma}

\begin{proof}
For every orbit $\rho$ of the action of $\Delta$ on atoms of $\mcA$ we choose a representative $A_\rho$. For every $A \in \rho$, $A \ne A_\rho$, we denote by $\delta(\rho,A)$ the element of $\Delta$ that induces on the atoms of $\mcA$ the transposition exchanging $A$ and $A_\rho$; we define similarly $\lambda(\rho,A)$ the element of $\Lambda$ inducing the transposition that maps $\Phi(A)$ to $\Phi(A_\rho)$. 

For each $\rho$, we can write 
$$\Phi(A_\rho)=  \bigsqcup_{B \in \text{atoms}(\mcA') \colon B \subseteq A_\rho} B' \ , $$
with $B'$ in the $G$-orbit of $B$, for each atom $B$ of $\mcA'$. 

We then define a finite Boolean algebra $\mcC'$ by setting as its atoms all $B'$ for $B$ an atom of $\mcA'$ contained in an $A_\rho$, as well as all $\lambda(\rho,A)(B')$ for $A \in A_\rho$. Then we obtain the desired $\Psi \colon \mcA' \to \mcC'$ as follows: first, we set $\Psi(B)=B'$ for all atoms $B$ of $\mcA'$ contained in an $A_\rho$; and whenever $A$ is an atom of $\mcA$ belonging some $\rho$, $A \ne A_\rho$, and $B$ is an atom of $\mcA'$ contained in $A$, we set $\Psi(B)= \lambda(\rho,A) \Psi \delta(\rho,A)(B)$.

We now need to describe the group $\Lambda'$.
In the remainder of this proof, the letter $\sigma$ will always stand for an orbit of the action of $\Delta$ on the atoms of $\mcA'$, and the letter $\xi$ for an orbit of the action of $\Delta'$ on the atoms of $\mcA'$. By definition, for all $\sigma$ there exists a unique $\xi$ such that $\sigma \subseteq \xi$.

We begin by picking, for each $\sigma$, an atom $E_\sigma$ of $\mcA'$ contained in $\sigma$. Then for each orbit $\xi$ we pick some $\sigma(\xi)$ such that $\sigma(\xi) \subseteq \xi$. 

We can find for all $\sigma,\xi$ such that $\sigma \subseteq \xi$ some $g(\xi,\sigma) \in G$ which maps $\Psi(E_{\sigma(\xi)})$ onto $\Psi(E_{\sigma})$ and is the identity elsewhere. For all $E \in \sigma$, $E \ne E_\sigma$ we denote by 
$\lambda(\sigma,E)$ the element of $\Lambda$ which induces on the atoms of $\mcC$ the transposition that maps $\Psi(E_\sigma)$ onto $\Psi(E)$. Then our desired group $\Lambda'$ is the group generated by 
$\{g(\xi,\sigma) \colon \sigma \subseteq \xi \} \cup \{\lambda(\sigma,E) \colon E \in \sigma, E \ne E_\sigma\}$.

By construction $(\mcC',\Lambda')$ is a unit system, and the action of $\Lambda'$ on the atoms of $\mcC'$ coincides with that of all permutations of the atoms of $\mcC'$ which stabilize all $\Psi(\xi)$. Thus $\Psi$ carries $\Delta'$-orbits onto $\Lambda'$-orbits and we have as desired $\Lambda'_{|\mcC'}= \Psi \Delta_{|\mcA'} \Psi^{-1}$.
\end{proof}

Now, we fix a full group $H$, and an ample group $\Lambda$ such that, for any $U,V \in \Clopen(X)$, one has
$$( \exists h \in H \ hU=V) \Leftrightarrow (\exists \lambda \in \Lambda \ \lambda U=V) $$
This condition amounts to stating that $\overline{H} = \overline {\Lambda}$; equivalently, $H$ is $\Lambda$-compatible and $\Lambda$ is $H$-compatible. 

\begin{prop}
There exists $g \in \Homeo(X)$ such that:
\begin{itemize}
    \item For every clopen $U$, there exists $h \in H$ such that $gU=hU$ (equivalently, $g \overline{H} g^{-1}= \overline{H}$).
    \item $g \Lambda g^{-1} \subseteq H$.    
\end{itemize}
\end{prop}

Applying this result to $H=\llbracket \alpha \rrbracket$ completes the proof of Theorem \ref{p:tame_gives_ample}.

\begin{proof}
We begin by fixing a refining sequence of finite unit systems $(\mcA_n, \Lambda_n)$ such that $\bigcup \mcA_n= \Clopen(X)$ and $\bigcup \Lambda_n= \Lambda$ (see \cite{Krieger1980} or \cite{mellerayInvariantMeasuresOrbit2021}).

Applying lemma \ref{l:Krieger} (with the role of the full group $G$ being played by $H$ at even steps, and by $\Lambda$ at odd steps), we can build by induction sequences of finite unit systems $(\mcC_n,\Delta_n)$ and $(\mcD_n,\Sigma_n)$, along with isomorphisms $\Phi_n \colon \mcC_n \to \mcD_n$ with the following properties:
\begin{enumerate}
\item For all $n$ $(\mcC_{n+1},\Delta_{n+1})$ refines $(\mcC_n,\Delta_n)$ and $(\mcD_{n+1},\Sigma_{n+1})$ refines $(\mcD_n,\Sigma_n)$;
\item For all $n$ $\Delta_n \subseteq \Lambda$ and $(\mcC_{2n},\Delta_{2n})$ refines $(\mcA_n,\Lambda_n)$;
\item For all $n$ $\Sigma_n \subseteq H$ and $\mcD_{2n+1}$ refines $\mcA_{n}$;
\item For all $n$, for all $A \in \mcC_n$ there exists $h \in H$ such that $\Phi(A)=hA$;
\item For all $n$, ${\Sigma_n}_{|\mcD_n}= \Phi_n {\Delta_n}_{|\mcC_n} \Phi_{n}^{-1}$;
\end{enumerate} 

This construction produces an automorphism $\Phi= \bigcup \Phi_n$ of $\Clopen(X)$, and by Stone duality there exists $g \in \Homeo(X)$ such that $gA=\Phi(A)$ for every clopen $A$. In particular $g$ is $H$-compatible, so $g\overline{H} g^{-1}= \overline{H}$.

We also built an ample subgroup $\Sigma = \bigcup \Sigma_n$ which is contained in $H$, and such that $g\Lambda g^{-1} = \Sigma$. The proof is complete.
\end{proof}

\section{Glossary}
For the reader's convenience we repeat below (in alphabetical order) the definitions of the monoid-theoretical notions that came up during the paper. Below $T$ is a commutative monoid with operation denoted by $+$ and neutral element $0$; we recall that $a \le b$ means that there is some $c$ such that $b=a+c$.

\begin{itemize}
\item $T$ is \emph{almost unperforated} if 
\[\forall x,y \in T \ \forall n \in \N \quad ((n+1)x \le  ny) \Rightarrow (x \le y) \]
\item $T$ is \emph{cancellative} if 
\[ \forall x,y,z \in T \quad (x+y=x+z) \Rightarrow y=z \]
\item $T$ is \emph{conical} if 
\[ \forall u,v \in T \quad (u+v=0) \Rightarrow (u=v=0) \]
This property is satisfied by every $T(\alpha)$.
\item  An element $x \in T$ is \emph{directly finite} if 
\[\forall y \in T \quad (y+x=x) \Rightarrow y=0 \]
\item An element $x \in T$ is a \emph{order unit} if 
\[\forall y \in T \  \exists n \in \N \quad  y \le nx \]
The type of a clopen subset $A$ is a order unit in $T(\alpha)$ iff translates of $A$ cover the whole space.
\item $T$ is a \emph{refinement monoid} if 
whenever $(a_i)_{1 \le i \le n}$ and $(b_j)_{1 \le j \le m}$ are elements of $T$ such that $\sum_{i=1}^n a_i = \sum_{j=1}^m b_j$ there exist elements $(c_{i,j})$ of $T$ such that $a_i=\sum_{j=1}^m c_{i,j}$ for all $i$ and $b_j = \sum_{i=1}^n c_{i,j}$ for all $j$.

\noindent Every $T(\alpha)$ is a refinement monoid.
\item $T$ is \emph{simple} if every nonzero element of $T$ is a order unit. This condition is satisfied in $T(\alpha)$ iff $\alpha$ is minimal.
\item $T$ is \emph{stably finite} if every element of $T$ is directly finite.
%\item $T$ is \emph{tame} if it is an inductive limit of finitely generated refinement monoids.
\item $T$ is \emph{unperforated} if 
\[\forall x,y \in T \ \forall n \in \N \setminus \{0\}  \quad nx = ny \Rightarrow x=y \]
\item Assume that $T$ is simple and conical, and let $x$ be a nonzero element of $T$. One says that $T$ has the \emph{weak comparability property} if 
\[ \forall a \ne 0 \ \exists k \in \N^* \  \forall b  \quad kb \le x \Rightarrow b \le a \]
(this definition does not depend on the choice of $x$ since we assume that every nonzero element is an order unit)
\end{itemize}

\bibliography{comparison}
\end{document}